\documentclass[]{amsart}

\usepackage{anyfontsize}
\usepackage{amssymb,bbm}
\usepackage{stmaryrd}
\SetSymbolFont{stmry}{bold}{U}{stmry}{m}{n}
\usepackage{braket}
\usepackage{mathtools}
\usepackage{enumerate}
\usepackage{proof}
\usepackage{booktabs}
\usepackage{rotating}
\usepackage{mathrsfs}  

\newtheorem{thm}{Theorem}[section]
\newtheorem{prop}[thm]{Proposition}
\newtheorem{lemma}[thm]{Lemma}
\newtheorem*{claim}{Claim}
\newtheorem{cor}[thm]{Corollary}
\theoremstyle{definition}
\newtheorem{defin}[thm]{Definition}

\renewcommand\P{{\mathcal P}}
\newcommand{\den}[2][]{\llbracket #2 \rrbracket_{#1}}
\renewcommand\phi\varphi 
\newcommand\dep{\mathtt{D}}
\newcommand\ruleskip{\qquad\qquad}

\newcommand\V{v}

\newcommand{\uBox}{\Box}

\makeatletter
\newcommand{\iland}{\mathrel{\vphantom{\land}\mathpalette\iland@\relax}}
\newcommand{\ilor}{\mathrel{\vphantom{\lor}\mathpalette\ilor@\relax}}
\newcommand{\ilnot}{\mathord{\vphantom{\lnot}\mathpalette\ilnot@\relax}}
\newcommand{\ibot}{\mathord{\vphantom{\bot}\mathpalette\ibot@\relax}}
\newcommand{\itop}{\mathord{\vphantom{\top}\mathpalette\itop@\relax}}
\newcommand{\dBox}{\mathord{\vphantom{\Box}\mathpalette\dBox@\relax}}

\newcommand{\iland@}[2]{\ooalign{\raisebox{.15ex}{\rotatebox[origin=c]{-90}{$\m@th#1\leqslant$}}}}
\newcommand{\ilor@}[2]{\ooalign{\raisebox{.15ex}{\rotatebox[origin=c]{-90}{$\m@th#1\geqslant$}}}}
\newcommand{\ibot@}[2]{\ooalign{$\m@th#1\LLbot$\cr\kern.2em$\m@th#1\LLbot$}}
\newcommand{\itop@}[2]{\ooalign{$\m@th#1\LLtop$\cr\kern.2em$\m@th#1\LLtop$}}
\newcommand{\ilnot@}[2]{\ooalign{\raisebox{.2ex}{$\m@th#1\LLnot$}\cr\raisebox{-.2ex}{$\m@th#1\LLnot$}}}
\newcommand{\dBox@}[2]{\ooalign{$\m@th#1\Box$\cr\kern.15em\raisebox{.5ex}{\scalebox{0.9}{$\m@th#1\raisebox{-.5pt}{\scalebox{0.7}{$<$}}$}}}}
\makeatother

\newcommand{\one}{\mathbbm 1}
\newcommand{\two}{\mathbbm 2}

\newcommand{\btop}{1}
\newcommand{\bbot}{0}
\newcommand{\bor}{+}
\newcommand{\band}{\cdot}
\newcommand{\bnot}{-}
\newcommand{\bleftrightarrow}{\leftrightarrow}

\let\LLnot\lnot
\let\LLor\lor
\let\LLand\land
\let\LLtop\top
\let\LLbot\bot

\newcommand{\etop}{\LLtop}
\newcommand{\ebot}{\LLbot}
\newcommand{\eor}{\LLor}
\newcommand{\eand}{\LLand}
\newcommand{\enot}{\LLnot}

\newcommand\Fm{\ensuremath{\mathtt{Fm}}}
\newcommand\Lb{\ensuremath{\mathtt{Lb}}}
\newcommand\of{\mathbin:}
\newcommand\ekv\leftrightarrow
\newcommand{\up}{\mathop{\uparrow}}
\newcommand{\dwn}{\mathop{\downarrow}}
\newcommand{\NB}{\mathord{\text{NB}}}

\newcommand{\slnot}{\mathord{{\sim}}}

\renewcommand{\oe}{\ensuremath{\eor\mathtt{E}}}
\newcommand{\oi}{\ensuremath{\eor\mathtt{I}}}
\renewcommand{\ae}{\ensuremath{\eand\mathtt{E}}}
\newcommand{\ai}{\ensuremath{\eand\mathtt{I}}}
\newcommand{\ioe}{\ensuremath{\mathord{\ilor}\mathtt{E}}}
\newcommand{\ioi}{\ensuremath{\mathord{\ilor}\mathtt{I}}}
\newcommand{\iae}{\ensuremath{\mathord{\iland}\mathtt{E}}}
\newcommand{\iai}{\ensuremath{\mathord{\iland}\mathtt{I}}}
\renewcommand{\ne}{\ensuremath{\enot\mathtt{E}}}
\renewcommand{\ni}{\ensuremath{\enot\mathtt{I}}}
\newcommand{\ine}{\ensuremath{\ilnot\mathtt{E}}}
\newcommand{\ini}{\ensuremath{\ilnot\mathtt{I}}}
\newcommand{\be}{\ensuremath{\ebot\mathtt{E}}}

\newcommand\taut{\ensuremath{\mathtt{taut}}}
\newcommand\sub{\ensuremath{\mathtt{sub}}}

\newcommand{\raa}{\ensuremath{\mathtt{RAA}}}

\begin{document}

\title{The Propositional Logic of Team Properties}

\author{Fredrik Engstr\"om}
\address{Department of Philosophy, Linguistics and Theory of Science,
University of Gothenburg, Box 100, 40530 G\"oteborg, Sweden}
\email{fredrik.engstrom@gu.se}

\author{Orvar Lorimer-Olsson}
\address{Department of Philosophy, Linguistics and Theory of Science,
University of Gothenburg, Box 100, 40530 G\"oteborg, Sweden}
\email{orvar.lorimer.olsson@gu.se}

\thanks{The authors contributed equally to this work.}

\date{\today}


\begin{abstract}
Since its introduction by Hodges and refinement by Väänänen, team semantic
constructions have been used to generate expressively enriched logics
preserving some desirable properties, such as compactness or decidability. By
contrast, these logics fail to be substitutional, limiting any algebraic
treatment, and rendering schematic uniform proof systems impossible. This
shortcoming can be attributed to \emph{the flatness principle}, commonly
adhered to when generating team semantics. Investigating the formation of
team semantics from algebraic semantics, and disregarding the flatness
principle, we present \emph{the Logic of Team Properties}, LTP, a
substitutional logic in which important propositional team logics are
axiomatisable as fragments. Starting from classical propositional logic and
Boolean algebras, we give a semantics for LTP by considering the algebras
that are powersets of Boolean algebras $B$, i.e., of the form $\P B$,
equipped with \emph{internal} (pointwise) and \emph{external}
(set-theoretic) connectives. Furthermore, we present a well-motivated sound
and complete labelled natural deduction system for LTP.
\end{abstract}

\maketitle

\section{Introduction}

Team semantics was invented by Hodges \cite{Hodges1997} to give the
Independence Friendly Logic (IF-Logic) of Hintikka and Sandu
\cite{Hintikka1989} a compositional semantics. Team semantics was later used
by Väänänen to define Dependence Logic \cite{Vaeaenaenen2007,vaananen2010dependence}, a formalism
extending first-order logic in which functional dependencies between variables
are explicitly expressed by atomic formulas. The intended meaning of these
atomic formulas $\dep(\bar x,y)$ is that the value of the variable $y$ is
functionally determined by the values of the finitely many variables $\bar x$.
Even though Dependence Logic uses only first-order quantifiers it can express
any existential second-order property or statement. 
This additional
expressive power comes from the semantic clauses for disjunction and existential
quantification: a disjunction may split a team into two subteams, while an
existential quantifier may choose, for each assignment in a team, a nonempty set
of possible witnesses. These choices amount to existential quantification over
relations or functions, which explains why second-order expressive
power arises from apparently first-order syntax.

Since this invention many logics based on team semantics have been introduced and
investigated, such as Independence Logic, propositional dependence logics and
modal dependence logics, see for example \cite{Vaeaenaenen2008,Yang2016,gradel2013dependence}.
We also relate this work to inquisitive logics \cite{Ciardelli2011}, which have separate origins.
These logics are formed by extending classical logic (or any intermediate
logic) with a notion of inquisitive propositions, and their standard
semantics have been found to be directly interpretable as a form of
propositional team semantics \cite{Yang2016}.

\subsection{Lifting Tarskian semantics to team semantics}

In the classical Tarskian semantics of first-order logic the denotation of a
formula, given a structure, is defined to be the set of all assignments that
satisfy the formula. Similarly, the denotation of a propositional formula is
the set of valuations satisfying the formula, and the denotation of a modal
logic formula, given a Kripke model, is the set of all worlds satisfying the
formula. Thus, the denotation of a formula in this classical setting is an
element of $\P X$, the powerset of the set $X$ of all assignments, all
valuations or all worlds, respectively.

Team semantics of first-order, propositional and modal logic lifts the
denotations of formulas to be sets of subsets of $X$, i.e., elements of $\P\P
X$, instead of elements of $\P X$. Thus, instead of asking if a single
assignment, valuation or world satisfies a formula, team semantics asks if a
\emph{set} of assignments, valuations or worlds satisfies a formula. Such sets
are called \emph{teams}.

This powerset lift makes it possible to define atoms and connectives that have
no corresponding definition in the classical setting. The dependence atom of
Propositional Dependence Logic is one such example: $$ X \vDash \dep
(\bar p,q) \quad\text{iff} \quad \forall s,s' \in X ( s(\bar p)=s'
(\bar p) \rightarrow s(q)=s'(q) ),$$ where $X$ is a set of propositional
valuations and $\bar p$ is a finite sequence of propositional variables.

In many standard presentations of team semantics, the semantics of the
classical fragment is required to agree with the usual pointwise semantics.
This requirement is often called the \emph{flatness principle}:
\begin{quote}
    A team satisfies a formula precisely when each of its individual members
    satisfies the formula in the usual classical sense.
\end{quote}
This principle applies only to formulas of the original language, before the
addition of dependence atoms or other genuinely team semantic connectives. For
such formulas, satisfaction by a single assignment, valuation, or world is
already defined in the underlying classical semantics.

In terms of denotations of formulas the flatness principle naturally
translates to the equation
\begin{equation}\label{eq:powerset}{}
    \den[h]{\phi} =\P \den[c]{\phi}
\end{equation} 
where $\den[h]{\phi}$ is the team-semantic denotation of $\phi$ ($h$ for
Hodges) and $\den[c]{\phi}$ is the ordinary Tarskian denotation of $\phi$ ($c$
for classical).\footnote{Here and in the remainder of the
introduction, notation of the form $\den{\varphi} $ is used
informally for the denotation of $\varphi$ under the indicated semantics; the formal
semantics of LTP is given in Definition~\ref{def_satisfaction}.}

Thus, for formulas in the original language, flatness identifies the
team-semantic denotation with the full powerset of the corresponding
classical denotation. This does more than determine the behaviour of the
classical connectives; it also restricts the possible denotations of atomic
formulas. Each atom is forced to denote the collection of all subteams of its
classical truth set, rather than an arbitrary collection of teams.

\subsection{Substitutionality and logics}

Dependence Logic and its variants have some desirable properties such as
compactness, Löwenheim-Skolem properties \cite{Vaeaenaenen2007} and that
first-order consequences of theories can be axiomatised \cite
{Kontinen2013} to name a few. But they are not substitutional; for example,
in Propositional Dependence Logic, $p \lor p \vDash p$ holds as usual.
However, substituting the dependence atom $\dep(p)$ for $p$ in the
entailment invalidates it: $$ \dep(p) \lor \dep(p) \nvDash \dep(p)$$ The
intuitive meaning of $\dep(p)$ is that the propositional variable $p$ is assigned a 
constant value. The team $X=\{s,s'\}$, where $s(p)=0$ and $s'(p)=1$ satisfies
$\dep(p)\lor\dep(p)$ as the team can be split into two singleton teams in
which the propositional variable $p$ has constant values. However, $X$ does
not satisfy $\dep(p)$ and the entailment is therefore invalid.

In general, a logic is substitutional if $\varphi \vDash \psi$ implies
$\phi[\sigma/p] \vDash \psi[\sigma/p]$, i.e., an entailment is not invalidated by
substituting a formula $\sigma$ for an atom $p$. Substitutionality was already
used by Bolzano to define the concept of validity. In \cite[§147]{Bolzano1837}
Bolzano defines \emph{universally valid propositions} as propositions with all
their variants true. In his terminology a variant is nothing but a
substitutional instance. Bolzano, and many after him, thus took
substitutionality not only as an important property for a logic, but as the
basic principle for the concept of logical validity.

Any attempt to give algebraic semantics to a logic that lacks substitutionality will necessarily fail in the strict sense used in abstract algebraic logic \cite{Font16}. In this setting, algebraic structures are treated with a uniform domain and classified by equational statements. Any such categorisation is necessarily closed under substitution. In non-substitutional logics, however,
formulas may exhibit type-sensitive behaviour, meaning that formulas with
syntactically similar structure may belong to semantically distinct
categories. Some approaches to deal with this for team logics that are \emph{downward closed} have been developed adjusting the algebraic semantics \cite{Bezhanishvili2021,Quadrellaro2020,Quadrellaro2021, Puncochar2017,Puncochar2021}. These are described in Section \ref{subsub: others approaches}. 

This challenge arises, for example, in propositional dependence logics due to
the way their semantics are defined via powerset constructions. Specifically,
the powerset lift of Tarskian semantics to team semantics (as described
in \eqref{eq:powerset}) introduces constraints on the denotations of atomic
formulas. Since these constraints are not preserved under arbitrary
substitution, the resulting logics generally fail to be substitutional.

In this paper we will therefore generalise away from using a specific lift.
Instead we will, in a certain sense, quantify over all possible lifts of the
atoms: 
Given a function $\mathcal L : \P X \to \P\P X$ that lifts the denotations
of \emph{atoms} from a Tarskian setting to a team-semantic setting, we
extend to all formulas $\phi$ compositionally to give team-semantic
denotations $\den[\mathcal L]{\phi}$ to all formulas $\phi$ in the logic. The
induced map $\phi \mapsto \den[\mathcal L]{\phi}$ is then the unique
homomorphism from the absolutely free term algebra of formulas into the
algebra of sets of teams determined by $\mathcal L$.
We then define
the entailment relation $\phi \vDash \psi$ to hold if $\den[\mathcal L]{\phi}
\subseteq \den[\mathcal L]{\psi}$ for all functions $\mathcal L : \P X \to
\P\P X$.

Thus, the function $\phi \mapsto \den[\mathcal L]{\phi}$ is a
homomorphism from the absolutely free term algebra of formulas into a specific
algebraic structure of sets of teams. Describing a logic by quantifying over 
all homomorphisms from a term language into an algebraic structure is exactly 
the starting point for constructing algebraic semantics.

\subsection{Lifting algebraic semantics}

The semantics of classical propositional logic can be defined in terms of
Boolean algebras. One may even say that classical propositional logic
\emph{is} the logic of Boolean algebras. The set of propositional formulas is
the term algebra generated by atoms using the Boolean operators $\bot, \lnot,
\lor$, and $\land$; and the semantics of propositional logic can be stated
using homomorphisms from this term algebra to a Boolean algebra. A formula is
a tautology if its image under any such homomorphism is the top element of the
Boolean algebra. This is the starting point when we define the Logic of Team Properties,
or LTP for short: Lifting the algebraic semantics for classical propositional
logic to the setting of teams. A team in this setting is nothing but a set of
elements of the Boolean algebra.

The connectives (and the corresponding operators on sets of teams) we are
interested in are the ordinary Boolean connectives $\ebot,\enot, \eand,\eor$
that correspond to the empty set, the complement, the union and the
intersection. We will call these connectives and the corresponding
operators \emph{external} Boolean connectives and operators. We will also add
the \emph{internal} connectives and operators that are defined by pointwise
application of  the operators (which will be denoted by  $+$, $\cdot$,
$-$) of the Boolean algebra, for
example the internal disjunction $\ilor$ is defined by $$A \ilor B = \set
{a \bor b | a \in A, b \in B},$$ where $A$ and $B$ are subsets of the Boolean
algebra. We will use blackboard boldface versions of the Boolean connectives
to denote these internal connectives: $\ibot,\ilnot,\ilor,\iland$. In
summary, the external connectives, $\enot, \eor,\eand$ are the set-theoretic
operators of complement, union and intersection; and the internal
connectives, $\ilnot, \ilor, \iland$ are defined by pointwise application of
the operators from the underlying Boolean algebra.

When the underlying Boolean algebra $B$ is $\P 2^{\mathbb N}$ with the
ordinary set-theoretic Boolean operators, the connective $\ilor$ corresponds
precisely to the ``splitjunction'' used in Dependence Logic, and $\iland$ is,
in the downward-closed setting of Dependence Logic, equivalent to the
conjunction in Dependence Logic.

In this setting, a team can be identified simply with an element of $B$, while
the semantic values of formulas then naturally become elements of $\P B$,
that is, sets of teams. Classical propositional logic may be seen as
reasoning about valuations, and admits an algebraic semantics in terms of the
Boolean algebra $B = \P(2^{\mathbb{N}})$ of sets of valuations. Under team
semantics, however, the basic semantic units are no longer single points,
valuations, or assignments, but teams.

Accordingly, once a Boolean algebra $B$ is taken as the space of teams,
formulas no longer denote single teams but collections of teams, i.e.,
elements of $\P B$. From this perspective, $\P B$ is not an auxiliary
construction but the natural semantic universe for reasoning about teams.
Moreover, by moving to the full powerset $\P B$, rather than to some
distinguished family of team-properties arising from a particular lift, we
obtain a logic that is genuinely about arbitrary collections of teams. In
this sense the present system is the propositional logic of team properties,
and is therefore naturally called the \emph{Logic of Team Properties} (LTP).

\subsection{Related constructions} 

An algebraic structure generated by the powerset of a Boolean algebra with
internal connectives is not new to mathematics and logic. Brink
\cite{Brink1984,Brink1986,Brink1993} contributes to these investigations and
calls them \emph{power algebras}, whereas Goldblatt calls them \emph{complex
algebras} \cite{Goldblatt1989} referring to a subset of an algebraic group as
a complex. It is also worth noting that these algebras are special cases of
\emph{Boolean algebras with operators} as described by J\'onsson and Tarski
\cite{Jonsson1993}. These play a notable role in the algebraic treatment of
modal logics, see \cite{Venema2007}.

In other related work, Priest has utilised the powerset lift in order to
investigate resulting families of \textit{plurivalent logics}
\cite{Priest2017}, an effort elaborated on by Humberstone in
\cite{Humberstone2014} coining the term \textit{power matrices} for the
resulting constructions. In these papers, semantics of multivalued logics is
lifted into evaluations on subsets of possible truth values, and the logical
connectives are interpreted in terms of pointwise operations. They do,
however, not include any connectives relating to the set-theoretic Boolean
operations on the powerset algebra.

In particular, in \cite{Goranko99} Goranko and Vakarelov use powersets of
Boolean algebras, referred to as \textit{hyperboolean algebras}, in order to
define what they call the \textit{hyperboolean modal logic}, HBML. This
construction treats a Boolean algebra, expressed as a partial order, as a
Kripke frame and utilises the Boolean structure to define modal operators.
The algebraic counterpart of the logic HBML is thus the powerset algebra of
the original Boolean algebra, with the modal operators defined as the internal
pointwise operations on the underlying algebra.




The logic of team properties that we define in this paper has different
motivation and origin than that of HBML, but building on the same class of
models. It is easy to see that the connectives of LTP and HBML are
interdefinable such that the two logics share validities and can in this
sense be viewed as having the same theorems. However, where Goranko and
Vakarelov only define HBML in terms of validity with a proof system
fundamentally structured around an elaborately defined \textit
{difference modality} and an \textit{only operator}, we define LTP for a full
entailment notion and present a labelled natural deduction system for which
the rules directly correspond to the basic connectives of the logic. Even so,
the correspondence between the two logics is strong enough for some important
properties of HBML presented in \cite{Goranko99} to also apply to LTP, see
Section \ref{sec:canonical}. In Section \ref{sub:modal_operators} we also
discuss Knudstorp's undecidability result for HBML, which transfers to
LTP as well.

\subsubsection{Related algebraic approaches in the downward-closed case} \label{subsub: others approaches}

The logic of team properties that we propose is a logic of arbitrary team
properties, that is, any predication on teams is viable to be represented as
a proposition in the logic. More work on algebraisation of propositional team
logics has been developed for logics for which the definable properties are
all \emph{downward-closed}, that is, if a team $X$ satisfies a property $P$,
so do all subteams $Y\subseteq X$. This includes both Propositional
dependence logic and the intermediate family of Inquisitive logics. A key
observation for this case is that the set of downward-closed collections of
teams (and more generally the collections of downward-closed subsets of
Boolean algebras) form a Heyting algebra under the appropriate operations,
making it possible to give algebraic semantics to these logics in terms of
classes of Heyting algebras with additional operators. However, as the logics
are not closed under substitution, some concessions need to be made
diverging from standard algebraic semantics.

Pun{\v{c}}och{\'a}{\v{r}} \cite{Puncochar2017,Puncochar2021} gives such
algebraic semantics for inquisitive logics by restricting the class of
homomorphisms into the algebras to those for which atomic formulas are mapped
to \emph{prime elements} of the Heyting algebras. In the specific models
formed from reinterpreting team semantics this exactly corresponds to
restricting atomic formulas to be represented by full powersets of teams in
line with the flatness principle.

Quadrellaro \cite{Quadrellaro2020,Quadrellaro2021} pushes the algebraisation
further by exploiting that, in virtue of the flatness principle, the
collection of  powersets of teams form a closed algebra under the operations
of the traditional logic that the team semantics conservatively extends
(for example, they form a Boolean algebra in Propositional Dependence Logic). This
allows him to describe the class of admissible homomorphisms as those where
atomic formulas are mapped into this particular type of substructure, making
it possible to define the semantics in terms of algebras together with a
unary predicate picking out an appropriate substructure. A limitation of this
approach is that the restricted class of homomorphisms can only map formulas
to a limited corner of the considered algebras referred to as \emph{the
core}, being the part of the algebra reachable through formulas from the
designated elements or substructure. Consequently the algebras of these
logics are left somewhat unspecified outside of this core.
 
In later work Bezhanishvili, Grilletti and Quadrellaro \cite{Bezhanishvili2021} 
found a more traditional way of dealing with this specific class of logics.
Starting from Quadrellaro's earlier semantics, the functional interpretation
of negation $(\neg)$ can be chosen so that its image is contained within the
designated substructure that identifies the original non-extended logic. This
allows for an axiomatisation of the properties of this substructure in terms
of formulas with negated atomic formulas. By then adding axioms only for
atomic formulas, satisfying a double-negation elimination ($\lnot \lnot p\to
p $), they manage to identify the family of intermediate inquisitive logics
in terms of \emph{negative variants} of intermediate logics. This makes it
possible to present nice axiomatic theories for these logics, but the
semantic counterpart still needs a demarcation of \emph{regular} elements
(satisfying $\lnot \lnot x = x$) for which the mapping of atomic formulas is
restricted.  
  
The setup and methods employed for the above-mentioned results are directly
guided by the specific logics they study and are set out to give algebraic
semantics for. As such, they lean strongly on the downward-closure and the
flatness principle for their construction and have to admit limitations in
the algebraisation connected to the non-substitutionality of the logics
discussed. 

Our approach is slightly different. We primarily attempt to describe a new
substitutional team logic fit for algebraic methods, but that is expressive
enough to define important existing team logics through axiomatisation.
Interestingly enough, the resulting axiomatisation we present bears some
resemblance to the axiomatisation for atoms given by Bezhanishvili et al.,
and it works essentially by identifying a class of appropriate homomorphisms
exactly in line with Pun{\v{c}}och{\'a}{\v{r}}'s semantic construction.  In
contrast to that work, we do not consider all intermediate logics as starting
points but focus only on those logics extending classical propositional
logic. At the same time, the valuational team logics we are able to
axiomatise are not limited to those with the downward-closure property. We
return to a more detailed comparison with the aforementioned results in the
concluding reflections in Section \ref{section:Conclusion}.

\subsection{Structure of this paper} 

In the next section we formally introduce the syntax and semantics of the
Logic of Team Properties, LTP. The semantics is defined in an algebraic manner in terms
of homomorphisms into algebraic structures based on Boolean algebras. 
In the same section we also introduce a labelled
natural deduction system for LTP. The formulas are decorated by labels and the
labels are themselves classical propositional formulas. By including rules in
the deduction system identifying classically equivalent formulas as equivalent
labels we establish the role of the labels as references to elements in a
Boolean algebra, and this paves the way for the completeness proof via
Lindenbaum--Tarski algebras presented in the next section.

Section 3 is devoted to prove the completeness of the natural deduction system
and the consequences that can be observed by a more careful investigation of
the proof. This section also includes results regarding non-canonicity and
adequacy of sets of Boolean algebras.

In Section 4 we introduce some important definable connectives in LTP. In
particular we define the \textit{strict negation} that, apart from being an
interesting type of negation, will be important in Section 6. We also define
a universal $\uBox$-modality in Definition~\ref{def:modal-operators}, which
plays an important role in Theorem~\ref{Def_to_ax}, where definable classes
of homomorphisms are converted into LTP-axiomatisations.

In Section 6 we finally utilise the strict negation to define classes of
homomorphisms, and give axioms expressing the strong propositional team
logic, PT$^+$ in \cite{Yang2017}, as a part of LTP. In this way we establish
the connection between LTP and the propositional team logics found in the
literature with semantics based on teams of valuations, here referred to as
\textit{valuational team semantics}. These results establish LTP as a well
motivated, substitutional, and expressively rich propositional team logic that
is highly relevant for a better understanding of team semantics for
propositional logics. 

In the final section of the paper we reflect generally on the construction
that we have presented, and discuss some further topics of investigation that
are implicated by our work.

\section{The Logic of Team Properties}\label{sec:logic}

Let us now define the Logic of Team Properties, LTP, which is the main object of study in
this paper. Let $B$ be a Boolean algebra; we allow $B$ to be the
trivial one-element Boolean algebra. To avoid confusion with the operators
on $\P B $ that we introduce below, we denote the operators and
constants of $B$ by $\bbot, \bnot, \bor,$ and $\band$.

The powerset $\P B$ is the domain of a Boolean algebra with
respect to the usual set-theoretic operators $(\emptyset,
\cdot^C, \cup, \cap)$, which we from now will denote by $\ebot$, $\enot$,
$\eor$, and $\eand$. These structures also have some additional natural
operators; the internal, or pointwise, operators: $\ibot$, $\ilnot$, $\ilor$,
and $\iland$. 

\begin{defin} 
Let $B$ be a Boolean algebra. We define the \emph{internal} operators in
$\P B$ as follows:
\begin{align*}
    \ibot &= \set{\bbot}, \\
    \ilnot X &= \set{\bnot a | a \in X},\\
    X \ilor Y &= \set{a \bor b | a \in X, b \in Y}, \text{ and}\\
    X \iland Y &=\set{a \band b | a \in X, b \in Y},
\end{align*} where $X,Y \subseteq B$.
\end{defin}

Formulas of LTP are elements in the term algebra, the absolutely free algebra,
generated by the propositional variables $P_0$, $P_1$, \dots\ using the
external Boolean connectives $\ebot$, $\enot$, $\eor$, $\eand$ and the
internal Boolean connectives $\ibot$, $\ilnot$, $\ilor$, $\iland$.\footnote{Observe that the symbols
$\ebot$, $\enot$, $\eor$, and $\eand$ are used at two levels: syntactically,
they denote the operations of the term algebra $\Fm$, and semantically, they
denote the corresponding set-theoretic operators on $\P B$. But note also that, from
the algebraic perspective, the connectives are themselves operations of the
(absolutely free) term algebra of formulas.}

\begin{defin}
The term algebra of formulas of LTP is denoted by $\Fm$ and defined by the following grammar:
$$\phi ::= \ebot \mathrel| \ibot \mathrel| P_i \mathrel| \enot \phi
\mathrel| \ilnot \phi \mathrel| \phi \eor \phi \mathrel| \phi \eand \phi
\mathrel| \phi \ilor \phi \mathrel| \phi \iland \phi. $$
\end{defin}

Entailment is defined as the subset relation for the images of the formulas
under arbitrary homomorphisms. This is in accordance with how entailment in
Dependence Logic is defined, in which $\Gamma$ entails $\psi$ if every team
that satisfies all formulas in $\Gamma$ also satisfies $\psi$.

\begin{defin}\label{def_entailment}
    Let $\Delta$ be a set of formulas. Then $\Delta \vDash \psi$ iff for all
    Boolean algebras $B$ and all homomorphisms $H:\Fm \to \P B$:
    $$\bigcap_{\phi \in\Delta} H(\phi) \subseteq H(\psi).$$
\end{defin}

Note that this definition of entailment corresponds to what, in the algebraic
logic literature, is called \emph{semilattice-based logics}, see for example
Section 7.2 in \cite{Font16} for details. 

The semantics of LTP admits an equivalent formulation in more traditional
terms. Instead of viewing formulas as interpreted via homomorphisms into $\P
B$, one can evaluate them directly on teams, in the style of standard team
semantics. We now present this alternative formulation.

\begin{defin}\label{def_satisfaction}
    Let $B$ be a Boolean algebra, $\V: \mathbb N \to \P B$ and $b\in B$. We define $B,\V,b \vDash \varphi$ by the following recursive definition. 
    \begin{itemize}
        \item $B,\V,b \nvDash \ebot$,
        \item $B,\V, b \vDash P_i $ if $b \in \V(i)$,
        \item $B,\V,b \vDash \enot \varphi$ if $B,\V,b \nvDash \varphi$,
        \item $B,\V,b \vDash \ilnot \varphi$ if $B,\V,\bnot b \vDash \varphi$,
        \item $B,\V,b \vDash \varphi \eor \psi$ if $B,\V,b \vDash \varphi$ or  $B,\V,b \vDash \psi$,
        \item $B,\V,b \vDash \varphi \eand \psi$ if $B,\V,b \vDash \varphi$ and $B,\V,b \vDash \psi$,
        \item $B,\V,b \vDash \varphi \ilor \psi$ if there are $b_1,b_2 \in B$ such that $b=b_1 \bor b_2$, $B,\V,b_1 \vDash \varphi$ and  $B,\V,b_2 \vDash \psi$, and
        \item $B,\V,b \vDash \varphi \iland \psi$ if there are $b_1,b_2 \in B$ such that $b=b_1 \band b_2$, $B,\V,b_1 \vDash \varphi$ and  $B,\V,b_2 \vDash \psi$.
\end{itemize}
\end{defin}

Thus LTP may be viewed in the usual way as a logic with models and a
satisfaction relation. The algebraic semantics presented above is then the
corresponding denotational formulation, obtained by assigning to each formula
the set of all teams that satisfy it.

Comparing the definition of the connectives in LTP with corresponding definitions in valuational team semantics, note that $\ilor$ corresponds to ``splitjunction'' and
$\enot$ to the classical negation used in propositional team logics \cite{Yang2016,Yang2017}. Observe
also that if $A$ and $B$ are downward-closed subsets of a Boolean algebra
with respect to its partial order $\leq$,\footnote{A subset $A$ of a set $X$ ordered by $\leq$
	is \emph{downward-closed} if for all elements $x\in X$, if $x\leq a$ for any
	$a\in A$ then also $x\in A$.} then $A \eand B =A\iland B$, and so both $\eand$
and $\iland$ can be seen as generalisations of the conjunction of Dependence
logic.

\begin{prop} \label{prop:hom-sat-equivalence}
Assume $B$ is a Boolean algebra, $\V:\mathbb N \to \P B$, $b \in B$, and
$H:\Fm \to \P B$ is a homomorphism such that $H(P_i)=\V(i)$. Then
$B,\V,b \vDash \varphi$ iff $b \in H(\varphi)$.
\end{prop}

\begin{proof} 
The proof proceeds by induction on $\varphi$.  The base case, where
$\phi$ is $P_i$, is immediate: $b \in H(P_i)$ if and only if $b \in \V
(i)$, which holds if and only if $B,\V,b \vDash P_i$.

The inductive steps are routine. We illustrate one direction in the case where
$\phi$ is $\psi \ilor \delta$. Suppose that
$B,\V,b \vDash \psi \ilor \delta$. Then there exist $b_1, b_2 \in B$ such
that $b_1 \bor b_2 = b$, with $B,\V,b_1 \vDash \psi$ and
$B,\V,b_2 \vDash \delta$. By the inductive hypothesis, $b_1 \in H(\psi)$ and
$b_2 \in H(\delta)$, and hence $b = b_1 \bor b_2 \in H(\psi \ilor \delta)$.
\end{proof}

The following proposition is an immediate consequence.

\begin{prop} \label{prop:entailment-satisfaction-equivalence}
$\Delta \vDash \psi$ iff for all Boolean algebras $B$, all
 $\V : \mathbb N \to \P B$, and all $b \in B$, if $B,\V,b \vDash \phi$ for all
 $\phi \in \Delta$, then $B,\V,b \vDash \psi$.
\end{prop}

\begin{proof} 
For the left-to-right direction, assume that $\Delta \vDash \psi$, and let
$B$, $\V$, and $b \in B$ be given. Suppose that $B,\V,b \vDash \phi$ for all
$\phi \in \Delta$. Let $H : \Fm \to \P B$ be the unique homomorphism such
that $H(P_i) = \V(i)$ for each $i$. 
By Proposition~\ref{prop:hom-sat-equivalence}, $b \in H(\phi)$ for all
$\phi \in \Delta$, and hence, by Definition~\ref{def_entailment},
$b \in H(\psi)$. Applying Proposition~\ref{prop:hom-sat-equivalence} once more,
we conclude that $B,\V,b \vDash \psi$.

The converse direction is proved similarly.
\end{proof}

The reader should note that when $B$ is the Boolean algebra of sets of
propositional valuations, i.e., $B=\P (2^{\mathbb N})$ and $b \in B$, then
$b$ is a team in the sense of standard propositional team semantics.
Moreover, if $\V(i) = \P´ \set{s \in 2^{\mathbb N} | s(i) =1 }$ then
$B,\V,b \vDash \varphi$ iff the team $b$ satisfies $\varphi$ in standard
Propositional Dependence Logic.\footnote{The precise translation between the
logical systems is given in Section~\ref{sec:valuational}, in particular in
Definition~\ref{def:PTplus} and Proposition~\ref{directThm}.}

The algebraic definition of LTP in terms of homomorphisms makes closure under
uniform substitutions immediate. Indeed, substitution instances of formulas
are precisely their images under homomorphisms $\sigma : \Fm \to \Fm$, and
closure under substitution follows directly from the definition of the
entailment relation.

\begin{thm}[Substitutionality]
    If $\Delta \vDash \psi$ then $\sigma(\Delta) \vDash \sigma(\psi)$ for
    every algebra homomorphism $\sigma:\Fm \to \Fm$.
\end{thm}
\begin{proof}
    This follows directly from the definitions, and the fact that if $H:\Fm \to \P B$
    is a homomorphism then so is $H\circ \sigma$.
\end{proof}

We will also use the following abbreviation for the \emph{internal top}.
\begin{defin}\label{def:internal-top}
The \emph{internal top} is the formula $\itop := \ilnot\ibot$.
\end{defin}

Thus, for every Boolean algebra $B$ and every homomorphism
$H:\Fm \to \P B$, $H(\itop)=\{1\}$.

\subsection{Natural deduction} \label{deduction-system}

To define a deductive system for LTP, we introduce \emph{labels} and
\emph{labelled formulas}. Labels are syntactic entities, just like formulas;
indeed, they are classical propositional formulas:

\begin{defin}
The term algebra of labels is denoted by $\Lb$ and defined by the following grammar:
$$\alpha ::= \bbot \mathrel| p_i \mathrel| \bnot \alpha \mathrel| \alpha \bor \alpha \mathrel|
\alpha \band \alpha.$$
The variables $p_i$ are called \emph{atomic labels}.
\end{defin}

The role of labels is to provide a syntactic representation of elements of the
underlying Boolean algebra. Although labels are purely syntactic objects,
they can be interpreted via homomorphisms as elements of a Boolean algebra,
while LTP formulas are interpreted as subsets of the same algebra. 

\begin{defin}
A \emph{labelled formula} is a pair of a label and a formula. We write a
labelled formula as $\alpha\of \phi$ where $\alpha$ is the label and $\phi$ the
formula. 
\end{defin}

The intended meaning of $\alpha\of\phi$ is that the element denoted by
$\alpha$ belongs to the set denoted by $\phi$. More precisely, given
homomorphisms $h:\Lb\to B$ and $H:\Fm\to \P B$, the labelled formula
$\alpha\of\phi$ is satisfied if $h(\alpha)\in H(\phi)$. In this way, labels
allow us to internalise the membership relation $h(\alpha)\in H(\phi)$ within
the proof system.

The labelled natural deduction system is motivated by this semantic picture:
its rules are designed to mirror directly the membership conditions defining
the internal connectives, whose semantic clauses involve decomposing an
element of the underlying Boolean algebra as a join or meet of two elements.
The corresponding labelled elimination rules make these decompositions
explicit in the proof system.

A label-free deductive system may well also be possible, and it would be of
independent interest to investigate whether a complete deductive system for
LTP can be formulated without labels. In particular, a Hilbert-style
axiomatisation of LTP might be obtainable. Such a system, however, would
serve a different purpose from the one pursued here: our aim is to give a
proof system whose rules track the algebraic semantics as closely as
possible, and which can plausibly be adapted to other team logics with
similar power-algebra semantics, including logics based on different initial
logics or equipped with different connectives.

We define entailment in the natural way:

\begin{defin}
    Let $\Gamma$ be a set of labelled formulas, then $ \Gamma \vDash \beta \of \psi$
    iff for all Boolean algebras $B$ and all homomorphisms $h: \Lb \to B$,
    $H:\Fm \to \P B$ if $h(\alpha) \in H(\phi)$ for all $\alpha\of \phi \in
    \Gamma$ then $h(\beta) \in H(\psi)$.
\end{defin}

We say that homomorphisms $H$ and $h$ make $\alpha \of \phi$ true if $h(\alpha) \in
H(\phi)$.

\begin{prop}
    $\Delta \vDash \psi$ iff $p\of \Delta \vDash p\of \psi$, where $p$ is an atomic
    label and $p\of \Delta = \set{p\of \phi | \phi \in \Delta}$.
\end{prop}

\begin{proof} 
The left-to-right implication is immediate from the definitions. Assume
$p\of \Delta \vDash p\of \psi$, and let $b \in B$. We let $h: \Lb \to B$ be
such that $h(p)=b$, then by the assumption we know that if $b \in H
(\delta)$ for all $\delta \in \Delta$ then $b \in H(\psi)$. Thus, $\bigcap_
{\delta \in\Delta} H(\delta) \subseteq H(\psi)$.
\end{proof}

We now present a sound and complete proof system for the relation
$\Gamma \vDash \beta \of \psi$, where $\Gamma$ is a set of labelled formulas. We
use the standard abbreviations $\alpha \bleftrightarrow \beta $ for 
$(\alpha \bor \bnot \beta )\band (\bnot \alpha  \bor \beta )$ and $\btop$ for $\bnot \bbot$.
Furthermore, we write $\alpha =\beta $ for the labelled formula $\alpha \bleftrightarrow
\beta \of \itop$. Note that whenever homomorphisms $H$ and $h$ make $\alpha =\beta $ true, it
follows that $h(\alpha )=h(\beta)$. Since labels are interpreted in Boolean algebras, we can use
classical propositional equivalences in calculations pertaining to labels.

\subsubsection{Rules for labels}

We first introduce the two rules governing labels:

$$\infer[\taut]{\beta \of  \itop}{\alpha_1 \of  \itop & \cdots & \alpha_k\of  \itop} \ruleskip
\infer[\sub]{\beta \of \phi}{\alpha = \beta  & \alpha \of \phi}$$

The rule $\taut$ is applicable only if $\alpha_1, \ldots, \alpha_k \vdash \beta$ in
classical propositional logic, i.e., if $\bnot (\alpha_1 \band \ldots
\band \alpha_k) \bor \beta$
is a tautology. Note that we allow for the special
case when $k=0$ in this rule.


\subsubsection{Rules for external Boolean connectives}

For the external Boolean connectives we use the usual natural deduction rules
for classical propositional logic, formulated for labelled formulas:

$$\infer[\ai]{\alpha \of \phi\eand \psi}{\alpha \of \phi & \alpha \of  \psi} \ruleskip
    \infer[\ae]{\alpha \of \phi}{\alpha \of \phi \eand \psi}\ruleskip
    \infer[\ae]{\alpha \of \psi}{\alpha \of \phi \eand \psi}$$

$$\infer[\oi]{\alpha \of \phi\eor \psi}{\alpha \of \phi} \ruleskip
    \infer[\oi]{\alpha \of \phi\eor \psi}{\alpha \of \psi} \ruleskip
    \infer[\oe]{\beta \of \sigma}{\alpha \of \phi \eor \psi & \infer*{\beta \of \sigma}{[\alpha\of \phi]} & \infer*{\beta \of \sigma}{[\alpha \of \psi]}}
$$

$$\infer[\ni]{\alpha \of \enot \phi}{\infer*{\beta \of \ebot}{[\alpha \of \phi]}} \ruleskip
    \infer[\ne]{\beta \of  \ebot}{\alpha \of \phi & \alpha \of \enot\phi} $$


$$\infer[\raa]{\alpha \of \phi}{\infer*{\beta \of \ebot}{[\alpha \of \enot\phi]}} \ruleskip \infer[\be]{\beta \of \phi}{\alpha \of \ebot} $$


\subsubsection{Rules for internal Boolean connectives}

Next we add rules for the internal connectives. The rule $\iae$ is modelled on
the fact that $b \in H(\phi \iland \psi)$ iff there are $c,d$ such that $c
\in H(\phi)$, $d \in H(\psi)$ and $b=c \band d$ holds in the Boolean algebra $B$. Similarly for
$\ioe$. 

$$\infer[\iai]{\alpha \band \beta \of  \phi \iland \psi}{\alpha \of \phi & \beta \of  \psi}
    \ruleskip
    \infer[\iae]{\beta \of \sigma}{\alpha \of \phi \iland \psi & \infer*{\beta \of \sigma}{\deduce{[\alpha =p\band
q]}{\deduce{[q\of \psi]}{[p\of \phi]}}}}$$ 

$$\infer[\ioi]{\alpha  \bor \beta \of  \phi \ilor \psi}{\alpha \of \phi & \beta \of  \psi}
    \ruleskip
    \infer[\ioe]{\beta \of \sigma}{\alpha \of \phi \ilor \psi & \infer*{\beta \of \sigma}{\deduce{[\alpha =p\bor q]}{\deduce{[q\of \psi]}{[p\of \phi]}}}}$$ 

$$\infer[\ini]{\bnot \alpha \of \ilnot \phi}{\alpha \of \phi} \ruleskip
    \infer[\ine]{\bnot \alpha \of \phi}{\alpha \of \ilnot \phi}$$

In $\iae$ and $\ioe$ the $p$ and $q$ are distinct atomic labels that
do not occur in any uncancelled assumptions, nor in $\alpha $ or $\beta $. These
elimination rules are the exact inverses of their corresponding introduction
rules: Anything that follows from the premises of an introduction rule can
also be derived from the conclusion of that introduction rule. Note also
that vacuous discharges are allowed in the elimination rules for $\ilor$ and
$\iland$.

Without the distinctness criteria in $\iae$ and $\ioe$, the system would not
be sound. Indeed, let $p$ and $q$ be atomic labels, then
$$ p=q\band q, q\of P, q\of Q \vdash p\of P\eand Q$$ 
is semantically sound, and witnessed by the derivation
$$
	\infer[\sub]{p\of P\eand Q}{
		\infer[\taut]{p= q}{p = q \band q } 
		&
		\infer[\ai]{q\of P\eand Q}{q\of P& q\of Q}}  
$$
Thus, without the distinctiveness condition in $\iae$, we could instantiate the
elimination rule with the same label $q$ twice and derive $p\of P\eand Q$ from
$p\of P \iland Q$. However, $$p\of P \iland Q \nvDash p\of P\eand Q$$ as shown by the following
example: Let $B=\two=\{\bbot,\btop\}$, the two-element Boolean algebra. If
$H:\Fm\to \P \two$ is a homomorphism, $H(P)=\{\btop\}$, and  $H
(Q)=\{\bbot\}$, then $H(P \eand Q)=\emptyset$. But $H(P \iland Q)=
\{\btop \band \bbot\}=\{\bbot\}$. Thus, for a label homomorphism
$h:\Lb \to \two$ such that $h(p)=\bbot$, clearly $h(p) \in H(P\iland Q)$,
but $h(p) \notin H(P\eand Q)$.

For an example of a non-trivial derivation of an entailment, see Figure
\ref{fig:derivation}. It is
 easy to see that the derivability relation satisfies the following lemma.

\begin{lemma}\label{lem:basicfacts} 
Let $\Gamma$ be a set of labelled formulas, then $\Gamma \vdash \alpha \of \phi$ iff
$\Gamma,\alpha \of\enot \phi \vdash \alpha \of \ebot$.
\end{lemma}

We now verify that the rules are sound with respect to the labelled semantics.

\begin{thm}[Soundness]
Let $\Gamma$ be a set of labelled formulas and assume $\Gamma \vdash
\alpha \of \phi$, then $\Gamma \vDash \alpha \of \phi$.
\end{thm}

\begin{proof} 
This follows from a straightforward induction on the construction of proof trees.
We will illustrate the proof with the case when the last rule of the proof
tree proving $\Gamma\vdash \beta \of \sigma$ is $\iae$: 
$$ \infer[\iae]{\beta\of \sigma}{\alpha \of \phi \iland \psi & \infer*
{\beta \of \sigma}{\deduce{[\alpha =p\band q]}{\deduce{[q\of \psi]}{
[p\of \phi]}}}}$$ 
The induction hypothesis gives us that $\Gamma \vDash \alpha \of \phi \iland \psi $
and $\Gamma, p\of \phi, q\of \psi ,\alpha =p\band q\vDash \beta \of \sigma$. If any
homomorphisms $H$ and $h$ make all labelled formulas in $\Gamma$ true then,
by the assumption, $h(\alpha) \in H(\phi \iland \psi)$. This means that there are
$c \in H(\phi)$ and $d \in H(\psi)$ such that $c\band d=h(\alpha)$. Let $h'$ be
the unique homomorphism that is like $h$ except that $h'(p)=c$ and $h'
(q)=d$. Then $h'$ and $H$ make $\Gamma$ true since $p$ and $q$ do not occur
in $\Gamma$. They also make $q\of \psi$ and $p\of \phi $ true. Also, $h'
(\alpha  \bleftrightarrow p\band q)= \btop$ and so, they also make $\alpha =p\band q$
true. By the induction hypothesis this means that $h'$ and $H$ make
$\beta \of \sigma$ true. Now, since $p$ and $q$ do not occur in $\beta$ we can safely
conclude that $h$ and $H$ also make $\beta\of \sigma$ true.
\end{proof}

\section{Completeness and adequacy}

Next we will prove the completeness of the proof system, but first we need some 
definitions and a few easy facts.

\begin{defin}
We say that $\Gamma$ is \emph{consistent} if $\Gamma \nvdash \bbot\of \ebot$.
\end{defin}

Observe that if $\Gamma \vdash \bbot\of \ebot$ then $\Gamma \vdash \alpha \of \phi$
for all labels $\alpha $ and all formulas $\phi$. Thus, $\Gamma$ is consistent iff
there is labelled formula $\alpha \of\phi$ such that $\Gamma \nvdash \alpha \of\phi$.

\begin{prop}\label{prop:max}
    If $\Gamma$ is consistent then so is either $\Gamma,\alpha \of \phi$ or
    $\Gamma, \alpha \of \enot \phi$.
\end{prop}
\begin{proof}
        This follows directly from the $\ni$ and $\ne$ rules.
\end{proof}

For the next proposition remember that if $\alpha $ and $\beta $ are labels then $\alpha =\beta $
denotes the labelled formula $\alpha  \bleftrightarrow \beta : \itop$.

\begin{prop}\label{prop:resp}
If $\Gamma,\alpha \of \phi\ilor \psi$ is consistent and $p$ and $q$ are atomic
labels not occurring in $\Gamma$, nor in $\alpha $, then $\Gamma,
\alpha \of \phi\ilor \psi,p\of \phi,q\of \psi,\alpha  = p \bor q$ is consistent. And
similarly for $\iland$.
\end{prop}
\begin{proof} 
This follows directly from the $\ioe$ rule, and the case for $\iland$ follows
in the same way from the $\iae$ rule.
\end{proof}

\begin{defin} 
$\Gamma$ is \emph{$\ilor$-saturated} if whenever $\alpha \of \phi\ilor
\psi \in \Gamma$ then there are labels $\beta_1$ and $\beta_2$ such that $\set{\beta_1\of
\phi,\beta_2\of \psi,\alpha  = \beta_1 \bor \beta_2} \subseteq
\Gamma$. The dual notion of \emph{$\iland$-saturation} is defined similarly.
 We say that $\Gamma$ is \emph{saturated} if it is both $\ilor$-saturated and
 $\iland$-saturated.
\end{defin}


\begin{lemma}
Let $\Gamma$ be a set of labelled formulas and $\Gamma'$ be the result of
renaming the atomic labels $p_i$ by $p_{2i}$. Then 
\begin{enumerate}
    \item $\Gamma$ is consistent iff $\Gamma'$ is, and
    \item $\Gamma$ is satisfiable iff $\Gamma'$ is.
\end{enumerate}
\end{lemma}
\begin{proof}(1) If $\Gamma \vdash \bbot \of \ebot$ then the same derivation with all
 atomic labels renamed shows that $\Gamma' \vdash \bbot \of \ebot$. For the other
 direction we need to be a little bit more careful as a proof of $\bbot \of \ebot$ may
 mention more atoms than those in $\Gamma'$. Let $\mathscr D$ be a derivation
 of $\bbot \of \ebot$ from $\Gamma'$, and replace all atoms $p_{2i}$ by $p_i$ and all
 atoms $p_{2i+1}$ by $p_{i+k}$ where $k$ is large enough, i.e., larger than
 all indices of atoms occurring in $\mathscr D$. Then the resulting
 derivation shows that $\Gamma \vdash \bbot \of \ebot$.

(2) Given a Boolean algebra $B$ and homomorphisms $H$ and $h$ such that $h
(\alpha) \in H(\varphi)$ for all $\alpha:\varphi \in \Gamma$ we can define $h'$ by
setting $h'(p_{2i}) = h(p_i)$ and $h'(p_{2i+1}) = \bbot \in B$. Then $B,H,h'$
satisfies $\Gamma'$. Also, if $B,H,h$ satisfies $\Gamma'$ then clearly for
$h'$ defined by $h'(p_{i})=h(p_{2i})$ we have that  $B,H,h'$ satisfies
$\Gamma$.
\end{proof}

\begin{thm}[Completeness]
    If $\Gamma \vDash \alpha \of \phi$ then $\Gamma \vdash \alpha \of \phi$.
\end{thm}
\begin{proof}
    Assume that $\Gamma \nvdash \alpha \of \phi$. By Lemma \ref{lem:basicfacts} the
    set $\Gamma_0 =\Gamma,\alpha \of \enot \phi$ is consistent. We will extend it to
    a maximal consistent saturated set $\Gamma^\ast$. By applying the previous
    lemma we may assume that there are infinitely many atomic labels that
    are not mentioned in $\Gamma_0$.

    We construct $\Gamma^\ast$ as the union of $\Gamma_n$ where each
    $\Gamma_n$ is a finite extension of $\Gamma_0$. First, we enumerate all
    labelled formulas and when constructing $\Gamma_{n+1}$ we pick the $n$th
    labelled formula $\beta  \of \psi$ and add either $\beta \of\psi$ or
    $\beta \of\enot\psi$. By Proposition \ref{prop:max} one of these is consistent
    with $\Gamma_n$.

    By Proposition \ref{prop:resp} we can assure that if we add a labelled
    formula $\beta \of \sigma \ilor \theta$ then we also add $p\of \sigma$, $q\of
    \theta$ and $\beta = p \bor q$ for some new atomic labels $p$ and
    $q$, and keeping $\Gamma_{n+1}$ consistent. 
    This construction assures that $\Gamma^\ast$ is maximal consistent and
    saturated.

    Now, let $T=\set{\alpha | \alpha \of  \itop \in \Gamma^\ast}$. Let $B$ be the
    Lindenbaum--Tarski Boolean algebra of labels over $T$, i.e., its elements
    are equivalence classes of labels under the relation of $T$-provable
    equivalence: 
    \[
        B = \Lb / \mathord\sim_T = \set{[\alpha]_T | \alpha \in \Lb},
    \]
    where $\alpha  \sim_T b $ iff $T \vdash \alpha \leftrightarrow \beta $ (in classical
    propositional logic) and $[\alpha]_T=\set{\beta \in \Lb | \alpha \sim_T \beta}$. Observe
    that $B$ is the trivial one-element Boolean algebra iff $T$ is the
    inconsistent theory, i.e., iff $\bbot \in T$.
    Define the homomorphisms $h: \Lb \to B$  and $H: \Fm \to \P B$ by
    \begin{gather*}
       h(p_i)=[p_i]_T \in B \text{ for atomic labels $p_i$, and}\\ H(P_i) = \set{h(\alpha) |
        \alpha\of P_i \in\Gamma^\ast} \text{ for atomic formulas $P_i$}.
    \end{gather*}

    {\sc Claim.} $h(\alpha) \in H(\varphi)$ iff $\alpha \of \varphi \in \Gamma^\ast$.

    The claim is proved by induction on formulas. The base case follows
    immediately from the definition of $H$ and the observation that 
    $$H(\ibot) =
    \{[\bbot]_T\}. $$  This is seen by taking $h(\beta) \in H(\ibot)$ and observing
    that $\beta \of \ibot
    \in \Gamma^\ast$ and so $\bnot \beta \of \itop
    \in \Gamma^\ast$ and, thus, $\bnot \beta  \in T$ and $\beta  \in [\bbot]_T$.
    \begin{itemize}
        \item 
    If $\varphi$ is $\psi \eor \sigma$, $\psi \eand \sigma$ or $\enot \psi$
    the induction step is straightforward, as for example $\alpha  \of  \psi \eand
    \sigma \in \Gamma^\ast$ iff $\alpha \of \psi \in \Gamma^\ast$ and $\alpha \of \sigma
    \in \Gamma^\ast$ and, by the induction hypothesis, this is equivalent to
    $h(\alpha ) \in H(\psi)$ and $h(\alpha ) \in H(\sigma)$, i.e., $h(\alpha ) \in H(\psi) \cap
    H(\sigma) = H(\psi\eand \sigma)$.

\item 
    For the case when $\varphi$ is $\ilnot \psi$ note that $\alpha \of  \ilnot \psi
    \in \Gamma^\ast$ iff $\bnot \alpha  \of \psi \in \Gamma^\ast$ iff $h(\bnot \alpha )
    \in H(\psi)$ iff $\bnot h(\alpha ) \in H(\psi)$ iff $h(\alpha) \in H(\ilnot \psi)$.

\item 
    When $\varphi$ is $\psi \ilor \sigma$ note that $\alpha \of  \psi \ilor \sigma \in
    \Gamma^\ast$ iff there are labels $\beta_1 $ and $\beta_2$ such that $\beta_1 \of \psi \in
    \Gamma^\ast$, $\beta_2\of \sigma \in \Gamma^\ast$ and $\alpha = \beta_1 \bor
    \beta_2 \in \Gamma^\ast$. By the induction hypothesis this is
    equivalent to $h(\beta_1) \in H(\psi)$, $h(\beta_2) \in H(\sigma)$ and $h(\alpha) = h(\beta_1\bor
    \beta_2) = h(\beta_1) \bor h(\beta_2)$. Thus this is equivalent to $h(\alpha) \in H(\psi \ilor
    \sigma)$.
\item 
    The case when $\varphi$ is $\psi \iland \sigma$ is treated similarly, ending the proof of the claim.
\end{itemize}
    It now follows immediately that for these choices
    of $B$, $H$, and $h$ we have $h(\alpha) \in H(\varphi)$ for all $\alpha \of\varphi \in
    \Gamma_0$ and thus $\Gamma \nvDash \alpha \of\phi$.
\end{proof}

\begin{cor}
    The logic LTP is compact.
\end{cor}
\begin{proof} 
This follows directly from soundness and completeness with respect to a finitary proof
system.
\end{proof}

\subsection{Canonical algebras}\label{sec:canonical}

The semantics of classical propositional logic can be given either in terms of
the two-element Boolean algebra or in terms of all Boolean algebras. In this
sense, the two-element Boolean algebra (and in fact every non-trivial Boolean
algebra) is \emph{canonical} for classical propositional logic. In LTP, the
situation is different, as we will see.

\begin{defin}
Let $\mathcal{X}$ be a class of Boolean algebras, then $\Delta \vDash_\mathcal
{X} \psi$ if for all $B \in \mathcal{X}$ and all homomorphisms $H:\Fm \to \P
B$: $$\bigcap_{\phi \in\Delta} H(\phi) \subseteq H(\psi).$$
\end{defin}

We write $\Delta \vDash_B \psi$ for $\Delta \vDash_{\{B\}} \psi$. This is a
restricted version of the entailment relation, $\vDash$, where the
quantification ranges over a restricted class $\mathcal{X}$ of Boolean algebras, or a single Boolean
algebra $B$, rather than over all Boolean algebras.

\begin{defin}
    A class $\mathcal{X}$ of Boolean algebras is \emph{adequate} for LTP if
    for every $\Delta$ and $\psi$ we have
    \[
    \Delta \vDash \psi \text{ iff } \Delta \vDash_\mathcal{X} \psi.
    \] 
   A single Boolean algebra $B$ is said to be \emph{canonical} for LTP if $\set{B}$
    is adequate for LTP.
\end{defin}

We first ask whether LTP admits a canonical algebra. The answer is negative,
as follows from the following simple observations. 

\begin{lemma}\label{lemma:trivial}
Let $\one$ denote the trivial one-element Boolean algebra and $\two$ the
two-element Boolean algebra.
    \begin{enumerate}
        \item\label{ett:label} $\vDash_B \ibot \text{ iff } B=\one$
        \item $\vDash_B \ibot \eor \ilnot\ibot \text{ iff } B \in \{\one,\two\}$
        \item If a formula $\phi$ contains no propositional variables then 
        $$\nvDash_B \varphi \text{ iff } \vDash_B \enot\ibot \eor (\ibot \iland \enot \phi).$$
    \end{enumerate}
\end{lemma}

\begin{proof}
\begin{enumerate}
    \item     $\vDash_B \ibot$ means that $B=\{0\}$ and thus that $B=\one$.
    \item Let $H$ be any homomorphism, then $H(\ibot \eor \ilnot\ibot)=\{0\} \cup  \{1\}$ which is $B$ iff $B\in \{ \one,\two \}$.
    \item  Since $\phi$ contains no propositional variables, $H(\phi)=H'(\phi)$ for all homomorphisms $H,H': \Fm \to \P B$. 
    Let $H$ be any such homomorphism. 

    If $\vDash_B\phi$ then $H(\enot \phi)=\emptyset$, and $H(\ibot \iland\enot \phi)=\emptyset$. 
    Thus, $H(\enot \ibot \eor (\ibot \iland \enot \phi) )=H(\enot \ibot)\neq B$. 
    On the other hand, if $\nvDash_B\phi$ then  $H(\ibot \iland\enot \phi)=\{0\}$, and 
    thus, $H(\enot \ibot \eor (\ibot \iland \enot \phi) )= (B\setminus \{0\}) \cup \{0\} = B$.\qedhere
\end{enumerate}
\end{proof}
\begin{thm}
No Boolean algebra is canonical for LTP.
\end{thm}
\begin{proof} 
It follows from Lemma \ref{lemma:trivial} that there are formulas $\phi_1$ and $\phi_2$ such
that $\vDash_B \phi_1$ iff $B\neq \one$ and $\vDash_B \phi_2$ iff
$B\neq \two$. This means that for any class of Boolean algebras $\mathcal
{X}$, if $\one \notin \mathcal{X}$ then $\vDash_\mathcal{X} \phi_1$. However,
$\nvDash \phi_1$, and thus for $\mathcal{X}$ to be adequate we must have that
$\one \in \mathcal{X}$. Similarly for $\phi_2$ and $\two$ and thus, any
adequate set $X$ includes both $\one$ and $\two$.
\end{proof} 

The non-adequacy of the class of finite Boolean algebras is established below
in Theorem~\ref{thm:finite_ba}. By contrast, the completeness argument above
already yields an adequate class of Boolean algebras: the countable
Lindenbaum--Tarski algebras.

\begin{thm}
    The class of countable (including finite) Boolean algebras is adequate for LTP.
\end{thm}
\begin{proof} 
The proof of the completeness theorem constructs a Boolean algebra $B$ as a
Lindenbaum--Tarski algebra over a countably infinite set of propositional
variables. More precisely, suppose that $\Delta \nvDash \phi$. Then
$p\of \Delta \nvDash p\of\phi$ and, hence, by soundness $p\of \Delta \nvdash
p\of\phi$. By the proof of the completeness theorem, there is a finite or
countable Boolean algebra $B$, and homomorphisms $h: \Lb \to B$ and
$H:\Fm \to \P B$ such that $h(p) \in H(\psi)$ for all $\psi \in \Delta$ and
$h(p) \notin H(\phi)$. Thus, $$\bigcap_{\psi \in \Delta} H(\psi) \nsubseteq H
(\phi)$$ and therefore $ \Delta \nvDash_{\{B\}} \phi$.
\end{proof}

\section{Definable connectives}\label{sec:definable}

We have defined an algebra based on external and internal Boolean connectives.

Using these connectives, we can define a number of additional constants and
operators. Some are familiar from ordinary propositional team logics and
modal logics, while others are introduced because they will be useful in
Section~\ref{sec:axiomatising}, where we relate LTP to propositional team
logics from the literature. It will be clear from the definitions we give
that there are plenty of other definable connectives; in particular, all
connectives defined for HBML in \cite{Goranko99} are definable in LTP. In
this paper we focus on the connectives that will be important for expressing
propositional team semantics, and leave further exploration for future work.
We name connectives after their interpretation in the intended semantics.

\subsection{External implication}

\begin{defin}
	The \emph{external implication} is the connective defined by
	\begin{equation*}
		\phi \to \psi := \enot \phi \eor \psi.
	\end{equation*}
\end{defin}

An easy argument shows that the deduction theorem holds in LTP.

\begin{thm}[Deduction theorem for LTP]\label{th:deduction}
	For all $\Delta \cup \set{\phi,\psi}\subseteq\Fm$
	$$\Delta, \phi  \vDash \psi \quad \text{iff} \quad \Delta \vDash \phi \to \psi.$$
\end{thm}
\begin{proof}
	By the definition of external implication we have, for every Boolean algebra
	$B$ and every homomorphism $H:\Fm \to \P B$,
	$$H(\phi \to \psi)=H(\enot\phi \eor \psi)=\bigl(B\setminus H(\phi)\bigr)\cup H(\psi).$$
	Assume first that $\Delta,\phi \vDash \psi$, and let
	$b\in \bigcap_{\delta\in\Delta} H(\delta)$. If $b\notin H(\phi)$, then
	$b\in H(\enot\phi)$, and hence $b\in H(\phi\to\psi)$. If instead
	$b\in H(\phi)$, then by the assumption $b\in H(\psi)$, and again
	$b\in H(\phi\to\psi)$. Thus $\Delta\vDash \phi\to\psi$.

	Conversely, assume that $\Delta\vDash \phi\to\psi$, and let
	$b\in \bigcap_{\delta\in\Delta}H(\delta)\cap H(\phi)$. By the assumption,
	$b\in H(\phi\to\psi)$. Since $b\in H(\phi)$, we have
	$b\notin H(\enot\phi)$, so $b\in H(\psi)$. Hence
	$\Delta,\phi\vDash \psi$.
\end{proof}

\subsection{Two constants}

We define two constants that will be used throughout the rest of the paper;
one is the \emph{external top}, and the other the \emph{not-bottom} constant
mimicking what was introduced in \cite{Yang2017} as the \emph{non-emptiness}, NE,
constant.

\begin{defin}
	The constants $\etop$ and $\NB$ are defined by
	\[
		\etop := \enot\ebot
		\quad\text{and}\quad
		\NB :=\enot \ibot.
	\]
	
\end{defin}

Recall from Definition~\ref{def:internal-top} that $\itop :=\ilnot\ibot$ is
the \emph{internal top}.

\begin{prop}\label{prop:constant_semantics}
	For every Boolean algebra $B$ and every homomorphism $H:\Fm\to\P B$,
	\[
		H(\etop)=B,
		\qquad
		H(\NB)=\set{x\in B \mid x\neq\bbot},\qquad\text{and}
		\qquad
		H(\itop)=\set{\btop}.
	\]
\end{prop}
\begin{proof}
	Since $H(\ebot)=\emptyset$, the first equality follows from the external
	interpretation of $\enot$. Since $H(\ibot)=\set{\bbot}$, the second equality
	follows in the same way. Finally, since $\itop=\ilnot\ibot$ and the internal
	negation is interpreted pointwise, $H(\itop)=\set{\bnot\bbot}=\set{\btop}$.
\end{proof}

Note, however, that the interpretations of some of these constants may
coincide for some algebras $B$, for example $\itop=\ibot$ for the trivial
Boolean algebra and $NB=\itop$ for the two-element Boolean algebra.

\subsection{Downward- and upward-closures}\label{def.clos} 

We define two useful closure operators that help identify when interpretations
of atoms are downward-closed, one of the key properties of standard
dependence logics.

\begin{defin}
	The \emph{downward-closure} and \emph{upward-closure} operators are defined
	by
	\[
		\dwn \phi := \phi \iland \etop
		\quad\text{and}\quad
		\up \phi := \phi \ilor \etop.
	\]
For elements $a$ and $b$ of a Boolean algebra, we write $a\leq b$ for the
Boolean algebra order, i.e., $a\band b=a$.
\end{defin}

\begin{prop}\label{prop:closure_semantics}
	For every Boolean algebra $B$, every homomorphism $H:\Fm\to\P B$, and every
	$\phi\in\Fm$,
	\[
		H(\dwn\phi)=\set{b\in B\mid \text{there exists }a\in H(\phi)\text{ such that }b\leq a},
	\]
	and
	\[
		H(\up\phi)=\set{b\in B\mid \text{there exists }a\in H(\phi)\text{ such that }a\leq b}.
	\]
\end{prop}
\begin{proof}
	By Proposition \ref{prop:constant_semantics}, $H(\etop)=B$. Hence
	\[
		H(\dwn\phi)=H(\phi\iland\etop)=
		\set{a\band c\mid a\in H(\phi), c\in B}.
	\]
	This is exactly the set of elements below some element of $H(\phi)$. The
	argument for $\up$ is the same, replacing meet by join.
\end{proof}
The following proposition validates that these are in fact closure operations and identifies some important relations to subsets corresponding to constants of our language.
\begin{prop}\label{prop:closure_properties}
	For every Boolean algebra $B$, every homomorphism $H:\Fm\to\P B$, and every
	$\phi\in\Fm$, the following hold:
	\begin{enumerate}
		\item $H(\dwn\dwn\phi)=H(\dwn\phi)$ and $H(\up\up\phi)=H(\up\phi)$.
		\item $H(\dwn\phi)=\emptyset$ iff $H(\phi)=\emptyset$, and otherwise
		$\bbot\in H(\dwn\phi)$.
		\item $H(\up\phi)=\emptyset$ iff $H(\phi)=\emptyset$, and otherwise
		$\btop\in H(\up\phi)$.
		\item $H(\dwn\phi)=B$ iff $\btop\in H(\phi)$.
		\item $H(\up\phi)=B$ iff $\bbot\in H(\phi)$.
	\end{enumerate}
	Moreover, $\etop$ and $\ebot$ are the only common fix-points of the two
	operators.
\end{prop}
\begin{proof}
	The first five statements follow immediately from Proposition
	\ref{prop:closure_semantics}. For example, if $H(\phi)$ is non-empty and
	$a\in H(\phi)$, then $\bbot\leq a$, so $\bbot\in H(\dwn\phi)$; and
	$H(\dwn\phi)=B$ holds exactly when every element is below some element of
	$H(\phi)$, which in a Boolean algebra is equivalent to $\btop\in H(\phi)$.
	The remaining cases are similar.

	If $A\subseteq B$ is both downward- and upward-closed and non-empty,
	then from any $a\in A$ we get both $\bbot\in A$ and $\btop\in A$, and hence
	$A=B$. Thus the only common fix-points are $\emptyset$ and $B$, i.e. the
	interpretations of $\ebot$ and $\etop$.
\end{proof}

We will now use these properties to define modal operators.

\subsection{Modal operators}\label{sub:modal_operators}

Any Boolean algebra can be regarded as a Kripke frame in which the
accessibility relation is the partial order of the Boolean algebra. Thus, a
Boolean algebra $B$ together with a homomorphism $H: \Fm \to \P B$ is
naturally a Kripke model in which $B,H,a \Vdash P_i$ iff $a \in H(P_i)$. With
this definition it is clear that $B,H,a \Vdash \phi$ iff $a \in H(\phi)$ for
all propositional formulas $\phi$, i.e., formulas that are built from atoms
using $\ebot,\enot,\eand$ and $\eor$.

From this perspective, we may in LTP, define some modal operators, in particular the necessity operator $\dBox$,
using the partial order of $B$ as the accessibility relation, and the
universal necessity operator $\uBox$.

\begin{defin}\label{def:modal-operators}
	The modal operators $\dBox$ and $\uBox$ are defined by 
	\[
		\dBox \phi := \enot\dwn\enot\phi 
		\quad\text{and}\quad
		\uBox \phi := \enot \up \dwn \enot \phi.
	\]
\end{defin}

\begin{prop}\label{prop:modal_semantics}
	For every Boolean algebra $B$, every homomorphism $H:\Fm\to\P B$, and every
	$\phi\in\Fm$,
	\[
		H(\dBox\phi)=\set{a\in B\mid \text{for all } b\geq a,\ b\in H(\phi)}.
	\]
	Moreover, $H(\uBox\phi)=B$ if $H(\phi)=B$, and $H(\uBox\phi)=\emptyset$
	otherwise.
\end{prop}
\begin{proof}
	By the definition of $\dBox$ and Proposition \ref{prop:closure_semantics}:
	$a\in H(\dBox\phi)$ iff $a\notin H(\dwn\enot\phi)$. This holds iff there is no
	$b\in B\setminus H(\phi)$ such that $a\leq b$, which is equivalent to saying
	that every $b\geq a$ belongs to $H(\phi)$.

	For $\uBox$, observe first that $H(\dwn\enot\phi)=\emptyset$ iff
	$H(\enot\phi)=\emptyset$, iff $H(\phi)=B$. In that case
	$H(\up\dwn\enot\phi)=\emptyset$, and hence $H(\uBox\phi)=B$. If instead
	$H(\phi)\neq B$, then $H(\enot\phi)$ is non-empty, so by Proposition
	\ref{prop:closure_properties}, $H(\up\dwn\enot\phi)=B$. Hence
	$H(\uBox\phi)=\emptyset$.
\end{proof}

This means that the semantics in LTP of $\dBox$ and $\uBox$ directly matches
their standard representation in Kripke semantics for modal logics in the
following sense.


\begin{defin}
	Let $B$ be a Boolean algebra, let $H:\Fm\to\P B$ be a homomorphism, and let
	$a\in B$. The relation $B,H,a\Vdash\phi$ is the usual modal satisfaction
	relation in the Kripke model whose set of worlds is $B$, whose accessibility
	relation for $\dBox$ is the order $\leq$ of $B$, and whose accessibility
	relation for $\uBox$ is the universal relation on $B$. Thus
	\[
		B,H,a\Vdash P_i \quad\text{iff}\quad a\in H(P_i),
	\]
	and the Boolean connectives are interpreted by the standard clauses:
	\begin{align*}
		B,H,a \nVdash \ebot \phantom{soii} \\
		B,H,a\Vdash \enot\phi \phantom{sii}&\quad\text{iff}\quad B,H,a\nVdash\phi,\\
		B,H,a\Vdash\phi\eand\psi &\quad\text{iff}\quad B,H,a\Vdash\phi
			\text{ and } B,H,a\Vdash\psi,\\
		B,H,a\Vdash\phi\eor\psi &\quad\text{iff}\quad B,H,a\Vdash\phi
			\text{ or } B,H,a\Vdash\psi.\\
		B,H,a\Vdash\phi\to\psi &\quad\text{iff}\quad B,H,a\nVdash\phi
			\text{ or } B,H,a\Vdash\psi.
	\end{align*}
	Moreover,
	\begin{align*}
		B,H,a\Vdash\dBox\phi &\quad\text{iff}\quad
			B,H,b\Vdash\phi\text{ for all }b\in B\text{ such that }a\leq b,\\
		B,H,a\Vdash\uBox\phi &\quad\text{iff}\quad
			B,H,b\Vdash\phi\text{ for all }b\in B.
	\end{align*}
\end{defin}

\begin{prop}\label{prop:modal}
	If $\phi$ is a formula in the language of modal logic with the two modal box
	operators, i.e., a formula built up from atoms using
	$\ebot,\enot,\eand,\eor,\to,\dBox,$ and $\uBox$; $B$ is a Boolean algebra and $H$ a
	homomorphism $\Fm \to \P B$, then
	\[
		B,H,a \Vdash \phi \quad\text{iff}\quad a \in H(\phi).
	\]
\end{prop}
\begin{proof}
	The proof is by induction on $\phi$. The Boolean cases are immediate from
	the interpretation of the external Boolean connectives. The case for $\dBox$
	is Proposition \ref{prop:modal_semantics}. The case for $\uBox$ follows from
	the second part of Proposition \ref{prop:modal_semantics}, since $\uBox$ is
	interpreted as the universal necessity operator.
\end{proof}

Indeed, the semantics of $\uBox$ are entirely independent of the internal
structure of the Boolean algebra $B$. In this context, $B$ functions simply
as a set of worlds, with $\uBox$ interpreted via the universal accessibility
relation. 
Hence the interpretation of formulas using only the external
Boolean connectives and $\uBox$ is exactly the usual Kripke semantics over
universal frames. The only relevant structure is the set of worlds together
with the valuation; duplicating a world, in the sense of adding a new world
with the same valuation as an existing one, does not change the truth of any
formula in this fragment. Since validity over universal frames is precisely
the modal logic S5, this $\uBox$-fragment of LTP corresponds to S5.

\begin{prop}\label{prop:S5_fragment}
	A formula built up from atoms using $\ebot,\enot, \eand, \eor,\to,$ and $\uBox$ is valid
	in LTP iff it is valid in the modal logic S5.
\end{prop}
\begin{proof}
By Proposition \ref{prop:modal_semantics}, $\uBox$ is interpreted as the
universal modality. Thus every LTP interpretation gives a universal Kripke model,
with the elements of the underlying Boolean algebra as worlds.

Conversely, if a formula in this fragment is not valid in S5, then it fails in
some finite universal Kripke model. By adding copies of worlds with the same
valuation, if necessary, we may assume that the number of worlds is the
cardinality of some finite Boolean algebra. Such duplication does not affect
truth of formulas in the $\uBox$-fragment. We may therefore identify the worlds
with the elements of that Boolean algebra, obtaining an LTP countermodel.
Hence the validities of this fragment are precisely the validities of universal
Kripke frames, i.e., S5.
\end{proof}

The $\uBox$ operation plays an important role in this paper in that it
facilitates the internalisation of classifications of global properties of
homomorphisms into formulas, see Section \ref{sec:def-class}.

Using the $\dBox$ operator and Proposition \ref{prop:modal} we can prove that
the finite Boolean algebras are not enough to define LTP. This proof follows
closely the proof of Theorem 3.12 in \cite{Goranko99}.

\begin{thm}\label{thm:finite_ba}
	The set of finite Boolean algebras is not adequate for LTP.
\end{thm}
\begin{proof}
	The following formula called \emph{Grzegorczyk's formula}
	\[
		\text{Grz:}\quad\dBox(\dBox(P\to \dBox P)\to P)\to P
	\]
	is known to be valid on all finite partially ordered Kripke frames, see \cite{Bull1984}.
	Thus, by Proposition \ref{prop:modal}, $\vDash_B \text{Grz}$ for all finite
	Boolean algebras $B$.

	On the other hand, if $B=\P \mathbb N$ and $H: \Fm \to \P B$ is such that
	$H(P)=B \setminus \{\emptyset,\{0,1\},\{0,1,2,3\},\ldots\}$, it is a
	straightforward calculation to check that $H(\text{Grz})=H(P) \neq B$.
	Therefore, $\nvDash_{\P {\mathbb N}} \text{Grz}$.
\end{proof}{}

Recently, Knudstorp \cite{Knudstorp2025} proved an undecidability result for the
related logic HBML. This result also transfers to LTP, and hence the set of valid
formulas of LTP,
$$	\set{\phi\in\Fm | \vDash \varphi},$$
is non-recursive. This also gives an alternative proof of Theorem
\ref{thm:finite_ba}: If the finite Boolean algebras were adequate for LTP,
then validity in LTP would be decidable. Indeed, by adequacy, any invalid
formula would have a countermodel over some finite Boolean algebra; since
finite Boolean algebras can be effectively enumerated, and validity over any
fixed finite Boolean algebra is decidable by checking the finitely many
relevant homomorphisms, invalidity would be recursively enumerable. On the
other hand, validity is recursively enumerable by the finite sound and
complete proof system for LTP. Hence validity would be decidable,
contradicting the non-recursiveness result.

\subsection{Strict negation}

When, in Section \ref{sec:axiomatising}, we relate LTP to other team semantics
described in the literature, we need to consider a third type of negation. We will
denote this negation by $\slnot$ and start by introducing it as an operation on an 
algebra  of the form $\P B$ where $B$ is Boolean.

\begin{defin}
  For a subset $A$ and an element $b$ of a Boolean algebra $B$ we say that $b$
	is \emph{separate} from $A$ if for all $a \in A$, $b\band a = \bbot$. We
	define $\slnot A$ as the set of all elements separate from $A$, i.e.,
	$$\slnot A = \set{ b \in B | \text{for all } a\in A, b\band a = \bbot }.$$
\end{defin}

This operation is definable as an operation in LTP, and we call it
\textit{strict negation}:

\begin{defin}\label{strict_negation}
	In LTP we define the unary operation $\slnot$ by $$ \slnot \phi= \enot\up(
	\dwn \phi\eand \NB ).$$ 
\end{defin}

\begin{prop}\label{slnot_form_to_subset}
	For every Boolean algebra $B$, every homomorphism $H:\Fm \to \P B$, and
	every formula $\phi\in\Fm$ we have $$ H(\slnot \phi) = \slnot H(\phi).$$
\end{prop}
\begin{proof}
	We first convince ourselves that for any algebra, homomorphism and formula
	as prescribed we have
	$$H(\dwn \phi \eand \NB)= \set{ a \in  B |  a \neq \bbot \text{ and } a\leq b
	\text{ for some } b\in H(\phi)}.$$
	Therefore $H(\up(\dwn \phi \eand \NB))$ is the set of elements of $B$ that
	have non-trivial intersection with some element in $H(\phi)$, i.e.,
	$$H(\up (\dwn \phi \eand \NB))= \set{ a\in B | a\band b \neq \bbot \text{ for
	some } b\in H(\phi)}.$$
	This is exactly the complement of $\slnot H(\phi)$, and thus $
	H(\slnot\phi)=  H (\enot\up( \dwn \phi\eand \NB ))= \slnot H(\phi)$.
\end{proof}

The purpose of introducing strict negation in this paper is to clarify the
relationship between LTP and more traditional forms of propositional team
semantics based on teams of valuations, see Section~\ref{sec:valuational}. At
the same time, we wish to highlight the following properties of $\slnot$ in
the framework of LTP.

\begin{prop}\label{prop_strict_neg}
	\begin{enumerate}
		\item $\vDash P \to \slnot \slnot P$ 
		\item $\nvDash \slnot \slnot P \to P $ 
		\item $\vDash \slnot  \slnot \slnot  P \rightarrow \slnot P$ 
		\item $\nvDash P \ilor \slnot P$
	\end{enumerate}  
\end{prop}
\begin{proof}
For a semantic proof of (1) we need to show that for any Boolean algebra $B$ and
any $A\subseteq B$ we have that $A\subseteq \slnot \slnot A$. In other words, if
$a\in A$ then for all $b\in \slnot A$ we have that $a\band b= \bbot$, which is
guaranteed by the definition of $\slnot A$. 

As a simple proof of statement (2), consider for any Boolean algebra a
homomorphism $H$ such that $H(P)= \emptyset$. Then $H(\slnot \slnot P) =
\set{\bbot}$. Clearly then $H(\slnot \slnot P)\nsubseteq H( P)$ and the
statement is proved.

For (3) assume for some $A\subseteq B $ that $b \in \slnot \slnot \slnot A$.
Then for all $a\in A$, by the semantic proof of (1) we have that $a\in\slnot
\slnot A$, and thus $b\band a = \bbot$, and thus $b\in \slnot A$. Hence $\slnot
\slnot \slnot A \subseteq \slnot A$. This proves statement (3).

For (4) let $B$ be the four element Boolean algebra $\{\bbot,a,b,\btop\}$ and $A=\{\bbot,a,b\}$. Then $\slnot A =\{\bbot\}$ and $A \cup \slnot A = A \neq B$. 
\end{proof}

The schema $P \ilor \slnot P$ will later be used in Section \ref
{sec:axiomatising} to axiomatise the propositional dependence logic 
$\text{PT}^+$. Statement (1) can also be proved directly in the deduction
system given in Section \ref{deduction-system}, or rather the corresponding
proof-theoretic statement $  p\of P \vdash p\of\slnot\slnot P$. Such a
derivation is presented in Figure \ref{fig:derivation} as an example of a
non-trivial deduction.

\begin{sidewaysfigure}
\vspace{.6\textwidth}
\hspace{10pt}
$
\infer[\ni_1]{p\of \slnot \slnot P}{
	\infer[\ioe_2]{p\of \ebot}{
		[p\of \up(\dwn \slnot P \eand \NB)]^1
		&
		\infer[\iae_3]{p\of \ebot}{
			[q \of \dwn\slnot P \eand \NB]^2
			& \hspace{-50pt}
			\infer[\be]{p\of \ebot}{
				\infer[\ne]{s\of \ebot}{
					[s\of \slnot P]^3
					&\hspace{-180pt}
					\infer[\taut /\sub]{s\of \up (\dwn P \eand \NB)}{
						\infer[\ioi]{(s\band t)\bor(s\band r)\bor (s\band \bnot (t\bor r))\of \up(\dwn P \eand \NB)}{
							\infer[\ai]{(s\band t )\bor (s\band r) \of \dwn P \eand \NB}{
								\infer[\taut / \sub]{(s\band t)\bor(s\band r)\of\dwn P}{
									\infer[\iai]{((s\band t)\bor r)\band s \of \dwn P}{
										\infer[\sub]{(s\band t)\bor r \of P}{
											\infer[\taut ]{p= (s\band t) \bor r}{
												[p=q\bor r]^2
												&
												[q= s\band t]^3}
											&
											p\of P}
										&
										\infer[\ni]{s \of \etop}{[s\of \ebot]}}}
								&
								\infer[\ni_4]{(s\band t)\bor(s\band r)\of \NB}{
									\infer[\ne]{(s\band t)\of \ebot}{
										\infer[\sub]{(s\band t )  \of \NB}{
											[q=s\band t]^3
											&
											\infer[\ae]{q\of \NB}{
												[q\of \dwn \slnot P \eand \NB]^2}}	
										& \hspace{-10pt}
										\infer[\taut/\sub]{(s\band t) \of \ibot}{
											\infer[\ine]{\bnot\bnot(s\band t) \of \ibot}{
												\infer[\taut]{\bnot (s\band t) \of \itop}{
													\infer[\ini]{\bnot((s\band t) \bor (s\band r)) \of \itop }{
														[(s\band t)\bor (s\band r) \of \ibot]^4}}}}}}}						
							&\hspace{-60pt}
							\infer[\ni]{s \band \bnot (t\bor r) \of \etop}{
								[s\band \bnot (t\bor r)\of \ebot]}}}}}}}}
$
\caption{Derivation showing that $p:P \vdash p:\slnot\slnot P$. Here
$\taut/\sub$ is a shorthand for a sequence of applications of $\taut$ and $\sub$ and $p,q,r,s,$ and $t$ are atomic labels.}
\label{fig:derivation}
\end{sidewaysfigure}

\section{Definable classes of homomorphisms}\label{sec:def-class}

The semantic definition of entailment in LTP is given as a universal
satisfaction of a property, evaluated independently for all homomorphisms and
for all Boolean algebras. To encode standard propositional dependence logic
within the LTP framework, we define specific logics by restricting the class
of homomorphisms under consideration. Furthermore, we introduce a notion
of \textit{definability} in LTP for such classes. By using the $\uBox$
operator we also obtain axiomatisations in LTP for logics of such definable
classes of homomorphisms.

\begin{defin}
	Let $B$ be a Boolean algebra and $H$ a homomorphism $\Fm \to \P B$. We
	define the \emph{local entailment} of $H$ denoted $\Delta \vDash_H \phi $ for LTP-formulas in
	the expected way:
	$$\Delta \vDash_H \phi \quad \text{ iff } \quad
	\bigcap_{\delta\in \Delta} H(\delta) \subseteq H(\phi).$$
	Given a class of homomorphisms $\mathcal{H}$ (not necessarily all to the
	same algebra), we define the corresponding entailment relation $$
	\Delta \vDash_{\mathcal H} \phi \quad \text{iff}\quad\text{for all }  H\in
	\mathcal{H}, \Delta \vDash_H \phi.$$
\end{defin}

A class of homomorphisms is definable if there is a set of formulas that are
valid precisely for the homomorphisms of that class.

\begin{defin}
	A class of homomorphisms $\mathcal{H}$ is \textit{definable} in LTP if there
	exists a set of formulas $\Pi$ such that for all homomorphisms $H: \Fm \to
	\P B$,
	$$H\in \mathcal{H} \quad \text{iff } \quad \vDash_H \pi \text{
		for all } \pi \in \Pi. $$
\end{defin}

Note that to axiomatise the logic of a class of homomorphisms it is not enough
to take a defining set of formulas $\Pi$ as axioms. The reason is that
definability imposes a global condition on homomorphisms, whereas the
entailment is a local condition. 

As a simple example, consider the class $\mathcal{H}_\one$ of homomorphisms into
the powerset of a one-element Boolean algebra. It is clear that this class is
defined by the formula $\ibot$: 
$$H\in\mathcal {H}_\one \quad \text {iff}\quad \vDash_H \ibot.$$ 
On the other hand we have that $\vDash_{\mathcal{H}_\one} \itop$ but
$\ibot \nvDash \itop$.

However, by using the universal modality $\uBox$, we can internalise global
conditions and ensure that $$x\in H(\uBox \eta) \quad \text{ iff} \quad
H(\eta)= B.$$ 
That is, for all $x \in B$, $$ x\in H(\uBox \eta)\quad \text{iff}
\quad \vDash_H \eta.$$ We can then conclude the following theorem:

\begin{thm} \label{Def_to_ax}
	Assume a class of homomorphisms $\mathcal{H}$ is defined by a set of formulas
	$\Pi$ and let $\uBox\Pi= \set{\uBox \pi | \pi\in \Pi}$. Then 
	$$ \Delta, \uBox \Pi \vDash \phi \quad \text{ iff } \quad
\Delta \vDash_{\mathcal H} \phi.$$
\end{thm}

This means that $\uBox \Pi$ serves as an axiomatisation in LTP of the logic given
by the restriction to homomorphisms of the class $\mathcal{H}$ that $\Pi$ defines.

\begin{proof}
	Assume first that $\Delta, \uBox\Pi \vDash \phi$. Let
	$H:\Fm\to \P B$ be a homomorphism in $\mathcal{H}$, and suppose that
	$b\in \bigcap_{\delta\in\Delta} H(\delta)$. Since $\mathcal{H}$ is defined
	by $\Pi$, we have $\vDash_H \pi$ for every $\pi\in\Pi$, that is,
	$H(\pi)=B$. Hence, by the semantics of $\uBox$, we have
	$H(\uBox\pi)=B$ for every $\pi\in\Pi$. It follows that
	\[
		b\in \bigcap_{\delta\in\Delta} H(\delta)
		\cap
		\bigcap_{\pi\in\Pi} H(\uBox\pi).
	\]
	By the assumption $\Delta,\uBox\Pi\vDash \phi$, we conclude that
	$b\in H(\phi)$. Thus $\Delta\vDash_H\phi$. Since $H\in\mathcal{H}$ was
	arbitrary, $\Delta\vDash_{\mathcal{H}}\phi$.

	Conversely, assume that $\Delta\vDash_{\mathcal{H}}\phi$. Let
	$H:\Fm\to \P B$ be any homomorphism, and suppose that
	\[
		b\in \bigcap_{\delta\in\Delta} H(\delta)
		\cap
		\bigcap_{\pi\in\Pi} H(\uBox\pi).
	\]
	Then $H(\uBox\pi)$ is non-empty for every $\pi\in\Pi$. Since, for every
	formula $\eta$, the set $H(\uBox\eta)$ is either $B$ or $\emptyset$, it
	follows that $H(\uBox\pi)=B$ for every $\pi\in\Pi$. Hence $H(\pi)=B$ for
	every $\pi\in\Pi$, i.e. $\vDash_H\pi$ for every $\pi\in\Pi$. Since
	$\Pi$ defines $\mathcal{H}$, this means that $H\in\mathcal{H}$. By
	$\Delta\vDash_{\mathcal{H}}\phi$, we therefore have
	$\Delta\vDash_H\phi$, and since $b\in\bigcap_{\delta\in\Delta}H(\delta)$,
	we get $b\in H(\phi)$. Thus $\Delta,\uBox\Pi\vDash\phi$.
\end{proof}

The preceding theorem shows that the universal modality $\uBox$ turns global
conditions on homomorphisms into ordinary assumptions of LTP. 
In particular, we can restrict the semantics to definable classes of
homomorphisms by means of axioms, rather than making such restrictions an
exterior part of the definition of the semantics as done in 
\cite{Quadrellaro2020,Quadrellaro2021,Bezhanishvili2021,Puncochar2017,Puncochar2021}.

Furthermore, the universal modality $\uBox$ paves the way towards a deeper
algebraic understanding of LTP. As defined in Definition \ref{def_entailment}, 
LTP is a semilattice-based logic using $(\P
(B),\subseteq)$ where $B$ is a Boolean algebra. These semilattices have
maximum elements, $B$, picked out by a term of the language: $\etop$. One can
therefore also consider the \emph{assertional companion logic} of LTP,
denoted $\mathrm{LTP}^{\etop}$; for details, see
\cite[Definition 7.36 and surrounding discussion]{Font16}. This is the logic
defined by the following entailment relation:

\begin{defin}
	Let $\Delta\cup \set{\psi}\subseteq \Fm$. Then $\Delta \vDash^\etop \psi$
	iff for all Boolean algebras  $B$ and all homomorphisms $H:\Fm \to \P B$:
	if $H(\delta)= H(\top)$ for all $\delta \in \Delta$ then $H(\psi)= H
	(\top)$. We denote this logic by $\text{LTP}^\etop$, the \emph
	{assertional companion of} LTP.
\end{defin}

Assertional logics are more central in the broader study of abstract algebraic
logic as presented by Font \cite{Font16} with many theories and results
directly applicable. They are also more tightly connected to Hilbert systems
from which one often can form an assertional algebraic semantics by the
Lindenbaum--Tarski algebras formed from quotienting a term algebra over
deductively closed theories.  

Importantly, semilattice-based and assertional companion logics share the same
theories, i.e., for LTP and $\text{LTP}^\top$ we have that $\vDash \psi$  if
and only if $\vDash^\etop \psi $. They do, however, differ in their entailment
notions. By the deduction theorem (Theorem \ref{th:deduction}) we can encode
any entailment with a finite set of premises
$\delta_1,\dots,\delta_n\vDash \psi$ as the theorem $\vDash (\delta_1\to
(\ldots \to (\delta_n \to \psi)\ldots )$, and thus investigate LTP also through the
study of $\text{LTP}^\etop$. Using the universal box $\uBox$ we can also
encode $\text{LTP}^\etop$ as theorems in LTP, by the following proposition
serving as a deduction theorem for $\text{LTP}^\etop$. 

\begin{prop}
	For all LTP-formulas $\phi$ and $\psi$, $\uBox \phi \vDash \psi$ iff for
	every Boolean algebra $B$ and every homomorphism $H : \Fm \to \P B$, $H
	(\phi)=H(\etop)$ implies $H(\psi)=H(\etop)$. Consequently, for $\text
	{LTP}^\etop$: $$\phi \vDash^\etop \psi \text{  if and only
	if } \vDash^\etop \uBox \phi \to \psi.$$
\end{prop}
\begin{proof}
  Since $H(\etop)=B$, and $H(\uBox\phi)=B$ iff $H(\phi)=B$, the statement
  $\uBox\phi \vDash \psi$ says exactly that, for every Boolean algebra $B$
  and every homomorphism $H:\Fm\to\P B$, if $H(\phi)=H(\etop)$, then $H
  (\psi)=H(\etop)$. The consequently part then follows from the deduction
  theorem of LTP (Theorem \ref{th:deduction}) and the clear fact that LTP and
  $\text{LTP}^\etop$ have the same theorems.
\end{proof} 

Many results and categorisations of properties outlined by Font \cite
{Font16} are applicable for both LTP and $\text{LTP}^\etop$ through these
observations. That said, further exploration in this direction will
nevertheless have to be left for future work.   

\section{Axiomatising valuational team semantics}\label{sec:axiomatising}

In this section we show how standard team semantics based on valuations relate
to the logic LTP. This construction essentially follows the construction
presented in \cite{LorimerOlsson2022} with some refinements and
generalisations. 

We begin by introducing the propositional team logic $\text{PT}^+$ through
valuational team semantics, where formulas are interpreted over sets of teams
of valuations. This logic is one of the more expressive propositional team logics in the literature, and many other logics are directly definable in it.
We then reformulate these semantics using the algebra $\P\P 2^
{\mathbb{N}}$ and a specific homomorphism $H_V$, aligning the description with the semantics
of LTP. This shows that LTP conservatively extends a logic weaker than $\text
{PT}^+$. Next, we identify the homomorphism $H_V$ as part of a class $\mathcal{H}_{\text
{PV}}$ of homomorphisms that are definable and axiomatisable in LTP. Finally,
we prove that $\text{PT}^+$ is axiomatised as a fragment of LTP via
the axioms for $\mathcal{H}_{\text{PV}}$, using two lemmas: one establishing
embedding-preserving properties of interpretations, and another showing that
all relevant homomorphisms can be represented via $H_V$ in $\P2^\mathbb{N}$.
The construction developed here has some striking similarities with the algebraic
treatment of downward-closed team logics referenced in Section 
\ref{subsub: others approaches}. We end the section by pointing out some of
these similarities, and also that some seemingly equivalent choices are in
fact importantly different.

\subsection{Valuational team semantics and $\text{PT}^+$}\label{sec:valuational}

In this section we will present the propositional dependence logic that we
will focus on in this paper. \emph{Strong propositional team logic}, $\text
{PT}^+$, is one of the strongest logics presented in
\cite{Yang2017} and we will define its semantics in a way that easily translates to the framework of LTP.

\begin{defin}\label{def:PTplus}
The set of formulas $\Fm_{\text{PT}^+}$ of \emph{strong propositional team logic},
$\text{PT}^+$,  is generated by the following grammar $$ \phi::= P_i
\mathrel|  \slnot P_i \mathrel|\ibot \mathrel| \NB\mathrel|
\phi \ilor \phi\mathrel|   \phi \land \phi \mathrel|  \phi \eor \phi $$ and
its valuational team semantics can be described by defining the denotations
for formulas $\den{ \phi } \subseteq \P 2^\mathbb{N} $
recursively for cases of the main connective as follows:
\begin{align*}
	\den{  P_i } &= \set{ X | \text {for all } s \in X, s(i)= 1}\\
	\den{ \slnot P_i } &= \set{ X | \text {for all } s \in X, s(i)= 0}\\
	\den{ \ibot } &= \set{ \emptyset}\\
	\den{ \NB } &= \set {X| X\neq \emptyset} \\
	\den{ \phi \ilor \psi } &= \set{ X\cup Y | X\in \den{
	\phi }, Y\in \den{ \psi }} \\
	\den{ \phi \eand \psi } &= \den{ \phi } \cap
	\den{ \psi }\\
	\den{ \phi \eor \psi } &= \den{ \phi } \cup
	\den{ \psi }\\
\end{align*}
The logical entailment of $\text{PT}^+$ is then defined as follows:
$$\Delta \vDash_{\text{PT}^+} \phi \quad \text{iff} \quad
\bigcap_{\delta \in \Delta}  \den{ \delta } \subseteq \den{
\phi }.$$
Elements $s\in 2^\mathbb{N}$ are viewed as \textit{valuations} for the set of
propositional variables, and sets of valuations $X\in \P 2^\mathbb{N}$ are
referred to as \textit{teams (of valuations)}.
\end{defin}

Note that the occurrence of $\slnot$ is restricted to propositional variables.
This follows the standard presentation of propositional team logics, where
strict negation is not a general formula-forming operation; see, for example,
\cite{Yang2016,Yang2017,Lueck2020}. This contrasts with LTP, in which
$\slnot \varphi$ is definable for arbitrary formulas.

There is no standard notation for the connectives in the literature, and we have
chosen notation that corresponds best to the notation for LTP. Table
\ref{translation} indicates the correspondence between our notation and
notation elsewhere. 
\begin{table}
	
\begin{tabular}{cccc}
	\toprule
	\cite{Yang2017} & \cite{Lueck2020} & \cite{Yang2022} & This paper \\
	\midrule
	$p_i$  		& $p_i$ 	& $p_i$ 	& $P_i$\\
	$\neg$ 		& $\neg$	& $\neg$	& $\slnot$ \\
	$\bot$ 		& $\bot$ 	& $\bot$	& $\ibot$ \\
	$\otimes  $ & $\vee $	& $\vee$	& $ \ilor $ \\
	$\wedge$ 	& $\wedge$	& $\wedge$  & $\eand $ \\
	$ \vee $ 	& $\ovee$	& $\ilor$   & $ \eor $ \\
	$\text{NE}$ & ~ 		& ~			& $\NB$ \\
	\bottomrule
\end{tabular}
\caption{Correspondence between notations in the current paper and other
relevant papers on propositional team logics.}\label{translation}
\end{table}

As defined, it is clear that for the set of formulas, including defined
connectives, we have that $\Fm_{\text{PT}^+}\subset
\Fm$ even though the logics are described using different semantics. In this
section we will find axioms in the language of LTP that axiomatise $\text{PT}^+$ 
in LTP. In the literature there are multiple weaker propositional team logics
described and studied, in particular propositional logics of dependence \cite
{Yang2016}. Many important logics, however, are given as, or are expressively
equivalent to, logics that can be given as fragments of $\text{PT}^+$. This
means that our result regarding the axiomatisation of $\text{PT}^+$ in LTP
will be directly applicable to these logics too. Table \ref
{list_of_logics} gives an overview of some of these weaker logics that are
directly definable as syntactic fragments of $\text{PT}^+$. The logic $\text{PD}^\vee$ 
is described in \cite{Yang2016}, and the others are described
in \cite{Yang2017}. Table \ref{ind_expr_logics} further lists logics that are
expressible in, or equivalent to, logics in Table 
 \ref{list_of_logics}, but not directly definable as a syntactic fragment. These
 results are presented in \cite{Yang2016,Yang2017} and we refer to these
 papers for more details.

\begin{table}
	\begin{tabular}{ll}
	\toprule
		Propositional team logics 								& Connectives \\
		\midrule 
		Classical propositional logic (CPL) 					&$\slnot P_i,\ibot,\ilor,\land$\\
		Strong classical propositional logic ($\text{CPL}^+$) 	&$\slnot P_i,\ibot,\ilor,\land,\NB$\\
		Propositional union closed logic ($\text{PU}$)		&$\slnot P_i,\ibot,\ilor,\land,\circledast$ \\
		Strong propositional union closed logic ($\text{PU}^+$)&$\slnot P_i,\ibot,\ilor,\land,\circledast,\NB$ \\
		Propositional dependence logic w. int. disj. ($\text{PD}^\vee$)& $ \slnot P_i,\ibot,\ilor, \land,\lor$ \\
		Propositional team logic ($\text{PT}$)				& $\slnot P_i,\ibot,\ilor,\land,\lor,\circledast$ \\
		Strong propositional team logic ($\text{PT}^+$)		& $\slnot P_i,\ibot,\ilor,\land,\lor,\NB$\\
		\bottomrule
	\end{tabular} 
\caption{Names and included connectives of logics described in
\cite{Yang2016} and \cite{Yang2017} as fragments of $\text{PT}^+$. The
connective $\circledast$ can be defined in $\text{PT}^+$ as $\phi \circledast
\psi := (\phi \eand \NB) \protect\ilor (\psi \eand \NB)$.  As in the standard
presentations of these logics, $\slnot P_i$ indicates that strict negation is
allowed only in front of propositional variables.}
\label{list_of_logics}
\end{table}
\begin{table}
	\begin{tabular}{ll}
		\toprule
		Propositional team logics								&  \\
		\midrule 
		Propositional inquisitive logic (InqL)					&equivalent to $\text{PD}^\vee$ \\
		Propositional intuitionistic dependence logic (PID)		&equivalent to $ \text{PD}^\vee$\\
		Propositional dependence logic (PD) 					&equivalent to $\text{PD}^\vee$\\
		Strong  propositional dependence logic (PD$^+$)			&expressible in $\text{PT}^+$\\
		Propositional independence logic (PI) 					&expressible in $\text{PT}$ \\
		Strong propositional independence logic (PI$^+$)		&expressible in $\text{PT}^+$ \\
		Propositional inclusion logic (PInc)					&expressible in $\text{PU}$\\
		Strong propositional inclusion logic (PInc$^+$)			&expressible in $\text{PU}^+$\\
		Full propositional team logic (FPT) 					&equivalent to $\text{PT}^+$ \\ 
		\bottomrule
	\end{tabular} 
	\caption{List of further team logics expressible in $\text{PT}^+$ or any of its fragments named in Table \ref{list_of_logics}. The first three appear in \cite{Yang2016} and the last seven appear in \cite{Yang2017}.}
	\label{ind_expr_logics} 
\end{table}

Observe that the denotations of valuational team semantics are elements of the
set $\P\P 2^{\mathbb{N}}$, which can be interpreted as a model of LTP by
interpreting $\P 2^{\mathbb{N}}$ as a Boolean algebra using the standard set
operations. We will refer to this as \textit{the valuation model}. With this
reading we can see that the interpretation of atomic formulas imposes a
specific homomorphism that maps every atomic formula to the set of teams for
which every member evaluates it to true.

\begin{defin}
Let $H_V$ denote \textit{the valuation homomorphism}, i.e., the unique homomorphism $H_V : \Fm \to \P\P
2^\mathbb{N}$ such that
\[ 
H_V(P_i) = \set { X\in \P 2^{\mathbb{N}} | \text{ for all } s\in X, s(i)= 1 }
= \P \set{ s\in 2^{\mathbb{N}} | s(i) = 1}.
\] 
\end{defin} 

The logic $\text{PT}^+$ corresponds to LTP restricted to the homomorphism
$H_V$, as established by the following proposition. The proof is
straightforward.

\begin{prop}\label{directThm}
	For all formulas $\Delta\cup \{\phi\} \subseteq \Fm_{\mathrm{PT}^+}$:
	\[
		 \Delta \vDash_{\mathrm{PT}^+} \phi \quad \text{ iff } \quad \Delta \vDash_{H_V} \phi.
	\]
\end{prop}

It follows directly that if $\Delta \vDash\phi$, then $ \Delta \vDash_{\mathrm{PT}^+} \phi$.

\subsection{Axiomatising a specific class of homomorphisms}\label{Ax-class-hom}

We observe that the valuation homomorphism $H_V$ has the following special
property. 
\[ 
H_V (P_i)= \P X \text{ for some } X \in \P 2^{\mathbb{N}}
\]
Algebraically speaking, every propositional variable is mapped to a non-empty
principal ideal of the Boolean algebra on $\P2^\mathbb{N}$, that is, a
subset $A$ of the Boolean algebra $B$ such that there is a maximal element
$a\in A$ generating $A$, i.e., such that $$A= \set{ b\in B | b\leq a}.$$ In
other words, $A$ is downward-closed and $\bigvee A\in A$.

We can in fact express the property of being a principal ideal in LTP in the
following formula akin to the excluded middle. Here $\slnot$ is the
strict negation of Definition~\ref{strict_negation}, and hence is available
for arbitrary LTP-formulas.

\begin{thm}\label{excluded_middle} 
	For all Boolean algebras $B$, all homomorphisms $ H: \Fm \to \P B$, and all
	formulas $\phi \in \Fm$ we have $$ \vDash_H \phi \ilor \slnot \phi
	\quad \text{iff}\quad H(\phi) \text{ is a principal ideal} $$
\end{thm}

The proof uses the following elementary observation, which will also be used
later. We state it separately for ease of reference.

\begin{lemma}\label{negated_principal}
	For a Boolean algebra $B$, if $A\subseteq B$ is a principal ideal with
	maximal element $a\in B$, then $\slnot A$ is a principal ideal with
	maximal element $\bnot a$.
\end{lemma} 

\begin{proof}[Proof of Lemma \ref{negated_principal}]
First note that for all $A\subseteq B$ we have that $\slnot A$ is downward-closed, 
since if $c\leq c'\in \slnot A$ then for all $b\in A$, $ b\band c \leq
b\band c' = \bbot$. What is left to show is that, if $a= \bigvee A$, then
$\bnot a = \bigvee \slnot A$. 

Clearly, for all $a'\in A$ we have $a'\band a = a'$. Hence
\[		
	\bnot a \band a'=\bnot a \band (a'\band a)=	a'\band(\bnot a\band a)	=\bbot,
\]
and therefore $\bnot a\in \slnot A$. 
Furthermore, if $b\in \slnot A$,
then $a\band b = \bbot$ and therefore $\bnot a = \bnot a \bor
(a \band b)= \bnot a \bor b$, i.e., $b\leq
\bnot a$. Thus, $\bnot a $ is an upper bound of $ \slnot A$ included in
the set, and therefore the set is a principal ideal, and $\bnot a = \bigvee
\slnot A$.
\end{proof}

\begin{proof}[Proof of Theorem \ref{excluded_middle}]
For one direction, assume $H(\phi \ilor \slnot \phi)= B$, we then need to
prove that  $H(\phi)$ is a principal ideal. First we establish that it is
downward-closed.  In search of a contradiction, assume that for some $a\in H
(\phi)$ we can find $b\in B$ such that $b\leq a$ and $b\notin H(\phi)$. We
want to show that then $b\notin H(\phi \ilor \slnot
\phi)$. If $b\in H(\phi\ilor \slnot \phi)$, then there exists $c\in H(\phi)$
and $d \in \slnot H(\phi)$ such that $c\bor d= b $. Since $b\leq a \in
H(\phi)$ we have that $$ b\band d \leq a\band d = \bbot$$ since $d\in
\slnot H(\phi)$. But then, since clearly $c\leq b$, we have $$b= b\band b
= (c\bor d ) \band b = ( c\band b )\bor (b\band d)= c\in H(\phi)$$ This is a
contradiction. We can therefore conclude, under the main assumption, that
$H(\phi)$ is downward-closed. Next we show that $\bigvee H(\phi)\in H(\phi)$.
By assumption we have that  $$\btop \in H(\phi \ilor \slnot
\phi)$$ Then there exists $c \in  H(\phi)$ and $d \in H(\slnot \phi )$ such
that $c\bor d=\btop$. Furthermore, for all $a \in H(\phi)$ we have that $$a =
\btop \band a = (c\bor d) \band a = c \band a, $$ since $d \band a  =
\bbot$ for all $a \in H( \phi) $. Consequently $a \leq c $ for all $a
\in H( \phi)$ and thus $ c$ is an upper bound for $H(\phi)$ included in
the set, i.e. $c  = \bigvee H(\phi)$. With downward closure established, this
also means that $H(\phi)$ is a principal ideal.

For the other direction, assume $H(\phi)$ is a principal ideal. Then there
exists $a \in B$ such that $a$ is the top element of $H(\phi)$. Then by Lemma
\ref{negated_principal} we see that  $\bnot a$ is the top element of the
principal ideal $H(\slnot \phi)$. Therefore, since $a \bor \bnot a =
\btop$ we have for every $b\in B$ that $$b= b \band (a\bor \bnot a)= (b\band a )
\bor (b\band \bnot a).$$ Being the top elements of the respective principal
ideals we observe that $$b\band a \in  H(\phi)\quad \text{and} \quad b
\band \bnot a \in H(\slnot \phi)$$ and conclude that  $ H(\phi  \ilor
\slnot \phi)= B$, in other words $H\vDash \phi \ilor \slnot \phi$.
\end{proof}

From this theorem we can directly conclude that
\[
\vDash_{H_V} P_i \ilor \slnot P_i \quad \text{for all } i \in
\mathbb{N}.
\]
We will see that this is the crucial categorisation of the homomorphisms that
relate to valuational team logics. We therefore identify the class defined by
these formulas, and the corresponding axiomatisation as discussed in Section
\ref{sec:def-class}.

\begin{defin}
	Let $\mathcal{H}_{\text{PV}}$ denote the class of homomorphisms defined by
	$$\set{ P_i\ilor \slnot P_i | i\in \mathbb{N}}.$$ We say that a
	homomorphism $H$ \textit{has principal  variables} iff
	$H\in\mathcal{H}_{\text{PV}}$. Furthermore, let \emph{the principal variable
	axioms} be the set
	\[
	\text{PVA}=\set{\Box ( P_i \ilor \slnot P_i) |i\in	\mathbb{N}}.
	\]
\end{defin}
 
 It follows directly from Theorem \ref{Def_to_ax} that, for all $\Delta,\set
 {\phi}\subseteq \Fm$,
 \[
 	 \Delta \vDash_{\mathcal{H}_{\text{PV}}} \phi \quad \text{if and only if} \quad \text{PVA},	\Delta\vDash \phi.
\]
It is evident that $H_V\in \mathcal{H}_{\text{PV}}$. 

\subsection{Axiomatisation of $\text{PT}^+$}

In this section we prove that the axioms PVA axiomatise $\text{PT}^+$ in the
sense of the following theorem.
\begin{thm}\label{thm:mainAx}
	For all $\Delta \cup \{ \phi \} \subseteq \Fm_{\mathrm{PT}^+}$ 
	\[ 
	\Delta \vDash_{\mathrm{PT}^+} \phi \quad \text{iff} \quad \mathrm{PVA}, \Delta \vDash \phi.
	\]
\end{thm}

The right-to-left direction follows directly from Theorem \ref{Def_to_ax} and
\ref{directThm}. To prove the other direction we first recall the fundamental
 result by Stone in the theory of Boolean algebras, see for example \cite
 [Chapter 22]{Halmos2009}.

\begin{thm}[Stone representation theorem]
\label{thm:completeBA}
	Every Boolean algebra can be embedded into a complete atomic Boolean
	algebra of the form $(\P S, \emptyset, \cdot^C, \cup, \cap)$ for some set
	$S$.
\end{thm}

Using this theorem we can establish Theorem \ref{thm:mainAx} in a two step
process. First showing that any homomorphism $H\in \mathcal{H}_{PV}$ can be
faithfully represented by a homomorphism in a Boolean algebra of subsets, and
then that any such homomorphism can be represented by the specific valuational
homomorphism $H_V: \Fm  \to \P \P 2^{\mathbb{N}}$. These steps are established
below in Lemma \ref{embedding_lemma} and Lemma \ref{f_rep_in_Val}
respectively. Both proofs are similar in structure, where the statement is
proven by induction over the complexity of formulas in $\mathrm{PT}^+$ for
which particular care is needed to handle the internal connective $\ilor$.

First we give a definition of an interval in a Boolean algebra. 

\begin{defin} 
		In a Boolean algebra $B$, if $a,b\in B$ and $a\leq b$, then let
		$[a,b]\subseteq B$ denote the set of elements between $a$ and $b$,
		that is $$[a,b]= \set{c\in B | a\leq c \text{ and } c\leq b  }. $$ We
		call this the closed interval of $a$ and $b$, and if $a=b$ we may
		write $[a]$ instead of $[a,a]$.
	\end{defin}  

\begin{lemma}\label{embedding_lemma}
	Let $B,B'$ be Boolean algebras, and $e:B\hookrightarrow B'$ an embedding
	of Boolean algebras. Then for each homomorphism $H:
	\Fm \to \P B $ with principal variables there  is a
	homomorphism $H': \Fm \to \P B'$  with principal variables
	such that for all formulas $\phi \in \Fm_{\text{PT}^+}$ $$ b\in H(\phi)
	\quad \text{ if and only if } \quad e(b) \in H'(\phi).$$
\end{lemma} 
\begin{proof}
	Assume $H: \Fm \to \P B$ has principal variables. Thus, for all $i$ there
	exists $b_i\in B$ such that $$H(P_i)= \set{ a\in B | a\leq b_i  }.$$ We
	then define $H': \Fm \to \P B'$ as the homomorphism that maps each
	propositional variable $P_i$ to the principal ideal of $e(b_i)$, i.e.,
	$$H'(P_i)= \set{ a'\in B'| a'\leq e(b_i)}.$$
	
	Clearly, $H'$ has principal variables, so what is left to show is that the
	equivalence in the theorem holds for all $b\in B$ and all formulas $\phi$
	of PT$^+$. In order to do so, we need an additional lemma best described
	using the notation of Boolean intervals.\footnote{The introduction of
	Boolean intervals in this context is inspired by their usage in
	\cite{Hella2023} in which minimal covers of Boolean intervals are used to
	define complexity measures for expressions in first-order team semantics.
	Even if our usage is noticeably different, the proof construction we
	present for forming included covering intervals is strongly related to the
	constructive methods of generating minimal interval covers for collections
	of teams satisfying a formula presented in that paper.}

\begin{claim} 
Let $B$ and $B'$ be Boolean algebras, let $e:B\hookrightarrow B'$ be a Boolean
algebra embedding, let $H:\Fm\to \P B$ be a homomorphism, and let
$H':\Fm\to \P B'$ be the homomorphism generated from $H$ as described above.
Then, for all $\phi\in \Fm_{\mathrm{PT}^+}$ and all $b'\in B'$, if $b'\in H'
(\phi)$, then there is some $b\in B$ such that $b'\leq e(b)$ and $[b',e
(b)]\subseteq H'(\phi)$.
\end{claim}

\begin{proof}[Proof of the claim]
	We prove this statement for all $b\in B$ by induction over the complexity
	of formulas. Let $b_i$ denote the generating element of the principal
	ideal $H(P_i)$.
	\begin{itemize}
		\item Assume $\phi = \ibot$. Then $b'\in H'(\ibot)$ only if $b'=  \bbot
		= e(\bbot)$, and thus $[b', e(\bbot)]= [\bbot] \subseteq H'(\ibot)$.
		\item Assume $\phi = \NB $. Then if $b'\in H'(\NB)$ clearly
		$[b',e(\btop)]\subseteq H'(\NB)$.
		\item Assume $\phi = P_i$. If $b'\in H'(P_i)$, then $b'\leq e(b_i)$
		and, since $H'(P_i)$ is a principal ideal, clearly
		$[b',e(b_i)]\subseteq  H'(P_i)$.
		\item Assume $\phi = \slnot P_i $. By Lemma \ref{negated_principal}
		and a similar argument as for the previous case we can conclude that
		if $b'\in H'(\slnot P_i)$ then $[b', e(\bnot b_i)] \subseteq H'(\slnot
		P_i)$.
		\item Assume $\phi = \psi \eand \chi$. If $b'\in H'(\phi)$, then
		$b'\in H'(\psi)$ and $b'\in H'(\chi)$. By induction, we may find
		$b_\psi,b_\chi\in B$ such that $[b',e(b_\psi)]\subseteq H'(\psi)$ and
		$[b',e(b_\chi)]\subseteq H'(\chi)$. Since $e$  is a
		homomorphism $$[b',e(b_\psi \band b_\chi)]\subseteq H'(\phi).$$
		\item $\phi = \psi \eor \chi$. If $b' \in H'(\phi)$, then $b'\in
		H'(\psi)$ or $b'\in H'(\chi)$, and without loss of generality we may
		assume the former. By induction hypothesis there is some $b_\psi \in
		B$ such that $[b',e(b_\psi)]\subseteq H'(\psi)$, and thus
		$[b',e(b_\psi)]\subseteq H'(\phi)$.
		\item $\phi = \psi \ilor \chi$. If $b'\in H'(\phi)$, then there is $
		c'\bor d'= b'$ such that $c'\in H'(\psi)$ and $d'\in H'(\chi)$. By
		induction hypothesis, we can then find $b_\psi, b_\chi \in B$ such
		that $[c',e(b_\psi)]\subseteq H'(\psi)$ and $[d',e(b_\chi)]\subseteq
		H'(\chi)$. Then since $b'= c'\bor d'$ and $e$ is a homomorphism, it is
		easy to see that $$H'(\phi)\supseteq \set{ p\bor q | p\in
		[c',e(b_\psi)], q \in [d', e(b_\chi)]} = [b', e(b_\psi \bor
		b_\chi)].$$
	\end{itemize}
	This concludes the proof of the claim.
\end{proof}

Now we are ready to finish the proof of the lemma. Again
this is achieved for all $b\in B$ by induction over the complexity of
formulas.
	\begin{itemize}
	\item Assume $\phi = \ibot$. First note that $b\in H(\ibot)$ if and only
	if $ b= \bbot\in B$ and  $b'\in H'(\ibot)$ if and only if $b'= \bbot \in
	B'$. Then note that $e(\bbot)= \bbot\in B'$ by $e$ being a homomorphism, and
	by $e$ being injective $e(b)= \bbot$ if and only if $b= \bbot\in B$. These
	observations suffice to prove the statement.
	\item Assume $\phi = \NB$. Since regardless of homomorphism and algebra
	$H(\NB)$ is the complement set of $H(\ibot)$, a similar argument proves
	the statement.
	\item Assume $\phi = P_i$. For all $b\in B$ we have that  $b\in H(\P_i)$
	if and only if $b\leq b_i$. Then by $e$ being an embedding this holds if
	and only if  $e(b)\leq e(b_i)$, which is equivalent to $e(b)\in H'(P_i)$.
	\item Assume $\phi = \slnot P_i$. By Lemma \ref{negated_principal}
	$H(\slnot P_i)$ is the principal ideal generated by $\bnot b_i$. Similarly,
	$H'(\slnot P_i) $  is the principal ideal generated by $\bnot e(b_i)$.
	Thus, a similar argument as for $\phi = P_i$ suffices.
	\item Assume $\phi = \psi \eand \chi$. Then $b\in H(\phi)$ if and only if
	$b\in H(\psi)$ and $b\in H(\chi)$. By induction we can conclude that this
	holds if and only if  $e(b)\in H'(\psi)$ and $e(b)\in H'(\chi)$ which
	holds if and only if  $e(b)\in H'(\phi)$.
	\item Assume $\phi = \psi \eor \chi$. The result follows in a similar 
	way as the previous case.
	\item Assume $\phi = \psi \ilor \chi$. If $b\in H(\phi)$, then there
	exists $c\bor d = b$ such that $c\in H(\psi)$ and $d\in H(\chi)$. By
	induction hypothesis we directly see that $e(b)= e(c)\bor e(d) \in
	H'(\phi)$. For the other direction, assume $e(b)\in H'(\phi)$. Then there
	exists $c'\bor d' = e(b)$ such that $c'\in H'(\psi) $ and $d'\in
	H'(\chi)$. Now by the interval lemma above we can find $b_\psi, b_\chi \in
	B$ such that $[c',e(b_\psi)]\subseteq H'(\psi)$ and
	$[d',e(b_\chi)]\subseteq H'(\chi)$. Now, let $c= b\band b_\psi$ then
	$$e(c)= e(b)\band e(b_\psi)= (c'\bor d')\band e(b_\psi)= c'\bor  (d' \band
	e(b_\psi))$$ so that clearly $e(c)\in [c', e(b_\psi)]\subseteq H'(\psi)$,
	and by induction hypothesis $c\in H(\psi)$. Similarly define $d= b\band
	b_\chi$ and conclude that $d\in H(\chi)$. 

	What is left to show is that $b=c\bor d$. Since
	$c'\leq e(b_\psi)$ and $d'\leq e(b_\chi)$, we have
	\[
		e(b)=c'\bor d'
		\leq e(b_\psi)\bor e(b_\chi)
		= e(b_\psi\bor b_\chi).
	\]
	As $e$ is an embedding, it reflects the order, and hence
	$b\leq b_\psi\bor b_\chi$. Therefore,
	\[
		c\bor d
		=
		(b\band b_\psi)\bor(b\band b_\chi)
		=
		b\band(b_\psi\bor b_\chi)
		=
		b.
	\]
	Thus $b=c\bor d$ with $c\in H(\psi)$ and $d\in H(\chi)$, so
	$b\in H(\psi\ilor\chi)=H(\phi)$.
\end{itemize} 
This concludes the proof of the lemma.
\end{proof}

\begin{lemma}\label{f_rep_in_Val}
For every Boolean algebra of the form $(\P S, \emptyset, \cdot^C, \cup ,\cap)$
and every homomorphism $H:\Fm \to \P \P S$ with principal
variables there is a mapping $f_H:S \to 2^{\mathbb{N}}$  such that for all
formulas $\phi \in \Fm_{\text{PT}^+}$ and all $X\in \P S$:
\[ 
X\in H(\phi) \quad \text{iff} \quad f_H^*(X) \in H_V(\phi),
\] 
where $f_H^*:\P S \to \P 2^{\mathbb{N}}$ is defined by $f_H^*(X) = \set{f_H(x) | x
\in X}$, and $H_V:\Fm \to \P\P 2^\mathbb{N}$ denotes the valuation
homomorphism.
\end{lemma}
\begin{proof}
	Assume $H:\Fm \to \P \P S$ has principal variables. Let $f: S\to
	2^{\mathbb{N}}$ be defined by
	\[
		f(s)(i)= \begin{cases} 1 & \text{ if } \{s\} \in H(P_i) \\ 
								0 & \text{ if } \{s\} \notin  H(P_i).
					\end{cases}
	\]

	\begin{claim} 
	For all $X,Y\in \P S$, $f^*(X\cup Y) = f^*(X)\cup f^*(Y)$.
	Furthermore, if $f^*(X) = U\cup V$, then there exist $Y,Z\in \P S$ such
	that $f^*(Y)= U$, $f^*(Z)= V$ and $Y\cup Z= X$.
	\end{claim}
\begin{proof}[Proof of claim]
	The first part of the claim is self-evident by the definition of $f^*$
	from $f$. For the second part, let $Y= \set{s\in X| f(s)\in U}$ and $Z=
	\set{s\in X | f(s) \in V}$.  Then every element of $X$ belongs to $Y\cup Z$,
	since $f(s)\in f^*(X)=U\cup V$, and by construction we have
	$f^*(Y)=U$ and $f^*(Z)=V$.
\end{proof}	
	We can now prove that the function $f$ is what we looked for in the lemma
	by induction over formulas $\phi\in	\Fm_{\text{PT}^+}$.
	
	We have four types of base cases:  $\ibot, \NB, P_i, \slnot P_i $:
	\begin{itemize}
		\item Assume $\phi = \ibot$. By definition $X\in H(\ibot)$ iff $X= \emptyset$, and
		since $f^*(X) = \emptyset $ if and only if $X= \emptyset$ this case is evident.
		\item Assume $\phi = \NB$. This case is proved by contraposition of the previous
		case.
		\item Assume $\phi = P_i$. By assumption $H(P_i)$ is a principal ideal in the Boolean
		algebra $\P S$. Hence, for all $X \in H(P_i)$, by downward closure we have
		for all $s\in X$ that $\{s\} \in H(P_i)$. By construction then
		$f^*(\{s\}) \in H_V (P_i)$ for all $s\in X$, and thus since $H_V$ has
		principal variables we find that $f^*(X)\in
		H_V(P_i)$. The opposite direction is proved with a similar chain of
		arguments.
		\item Assume $\phi = \slnot P_i$. Then by Lemma
		\ref{negated_principal} we assert that for all $H\in \mathcal{H}_{PV}$
		we have that $H(\slnot P_i)$ is a principal ideal. The proof is then
		similar to the previous case.
	\end{itemize}

	For the induction step we have three cases for the main connectives:
	$\eand ,\eor,\ilor$
	\begin{itemize}
		\item Assume $\phi = \psi \eand \chi$. $X\in H(\phi)$ by definition if
		and only if $X\in H(\psi)$ and $X\in H(\chi)$. By induction
		hypothesis, we can conclude that is the case if and only if $f^*(X)\in
		H_V(\psi)$ and $f^*(X)\in H_V(\chi)$, which is equivalent to stating
		that $f^*(X) \in H_V(\phi)$.
		\item Assume $\phi = \psi \eor \chi$. The proof is similar to the previous case.
		\item Assume  $\phi = \psi \ilor \chi$. For one direction, assume $X \in
		H(\phi)$. Then there exist $Y\in H(\psi)$ and $Z\in H(\chi)$ such
		that  $X= Y\cup Z$. By induction hypothesis $f^*(Y)\in H_V(\psi)$ and
		$f^*(Z)\in H_V(\chi)$. By the first part of the Claim about $f^*$ we
		see that $f^*(X)= f^*(Y)\cup f^*(Z)$ and thus $f^*(X)\in H_V(\phi)$.
		
		For the other direction, assume $f^*(X) \in H_V(\phi)$.  Then there
		exist $U\in H_V(\psi), V\in H_V(\chi)$ such that $f^*(X)=U\cup V$.
		Then by the second part of the Claim about $f^*$, there exists $Z, Y$
		such that $f^*(Z)=U, f^*(Y)= V$ and $Z\cup Y= X$. By the induction
		hypothesis  $Z\in H(\psi)$ and $Y\in H(\chi)$.  We conclude
		that $X\in H(\phi)$.
	\end{itemize}
	This concludes the proof of the lemma.
\end{proof}

We are now ready to prove Theorem \ref{thm:mainAx}.

\begin{proof}[Proof of Theorem \ref{thm:mainAx}] 
	By applying Theorem \ref{Def_to_ax} the statement of Theorem \ref{thm:mainAx} is equivalent to the statement that for all $\Delta \cup \{
	\phi \} \subseteq \Fm_{\text{PT}^+}$ 
	\[
	\Delta \vDash_{\text{PT}^+} \phi \quad \text{iff} \quad
	\Delta \vDash_{\mathcal{H}_{\text{PV}}} \phi.
	\]
	The right-to-left direction follows directly from Proposition \ref{directThm} and the fact
	that $H_V\in\mathcal{H}_{\text{PV}} $. 

	The other direction is proved by contraposition. Assume
	$\mathcal{H}_{\text{PV}} :  \Delta \nvDash \phi$ and thus  $\text{LTP} :
	\text{PVA}, \Delta \nvDash \phi$.  Then there is some Boolean algebra $B$ and 
	homomorphism $H: \Fm \to  \P B$ such that $ H:
	\text{PVA},\Delta \nvDash \phi.$ In other words, $H\in
	\mathcal{H}_{\text{PV}}$ and there exists an element $ b\in B $ such that
	\[
	b \in \bigcap_{\delta\in \Delta} H(\delta), \text{ but } b \notin H(\phi).
	\]  
	By Theorem \ref{thm:completeBA}, we can then find an embedding $e :B \hookrightarrow \P{S}$ of $B$ into a complete atomic Boolean algebra of the form $(\P S, \emptyset, \cdot^C, \cup
	,\cap)$. Then by Lemma \ref{embedding_lemma} and \ref{f_rep_in_Val}  we can find a homomorphism $H':\Fm\to \P\P S$, and a mapping $f_{H'}: S\to 2^\mathbb{N}$  such that 
	\[
	f_{H'}^*(e(b)) \in \bigcap_{\delta\in\Delta} H_V(\delta), \text{ but }
	f_{H'}^*(e(b))\notin H_V(\phi).
	\]
	Thus, $H_V :\Delta \nvDash \phi$ and by Theorem \ref{directThm}: 
	\[
	\Delta \nvDash_{\text{PT}^+} \phi.
	\] 
	This finalises the proof of Theorem	\ref{thm:mainAx}.
\end{proof}

Observe that to evaluate $\text{PT}^+$ it is sufficient to consider the
valuational algebra $\P 2^\mathbb{N}$, i.e., for all  $\Delta, \{\phi\}
\subseteq\Fm_{\text{PT}^+}$ we have 
\[
\Delta \vDash_{\text{PT}^+} \phi \quad\text{iff}\quad 
\text{PVA},\Delta \vDash_{\P 2^\mathbb{N}} \phi.
\]
In this sense $\P 2^\mathbb{N}$ can be seen as canonical for $\text{PT}^+$. It
is, however, \textit{not} canonical for the logic in LTP axiomatised by PVA,
since the canonicity only holds when the formulas are restricted to the
language $\Fm_{\text{PT}^+}$.

We have proven the main correspondence result for the strong propositional
team logic $\mathrm{PT}^+$ by isolating, in the axiom set $\mathrm{PVA}$, the
crucial difference between valuational team semantics and the more general LTP
semantics. Consequently, these same axioms suffice to axiomatise all team
logics expressible in $\text{PT}^+$ and the following corollary is
immediate.

\begin{cor}
	For every logic $L$ listed in Table~\ref{list_of_logics} and all
	$\Delta \cup \{\phi\} \subseteq \Fm_L$,
	\[
		\Delta \vDash_L \phi
		\quad \text{iff} \quad
		\mathrm{PVA}, \Delta \vDash \phi.
	\]
	Here $\vDash_L$ denotes entailment in the logic $L$ with respect to
	standard team semantics. Consequently, all logics in Table \ref{ind_expr_logics} are indirectly expressible in LTP in the sense of \cite{Yang2017}.
\end{cor}

\subsection{Comparing to existing algebraisations}

After developing our encoding of valuational team semantics in LTP, we now
compare the resulting construction with the related approaches outlined in
Section \ref{subsub: others approaches}. All of these constructions present
semantics by specifying classes of algebras, but then externally
restricting the classes of homomorphisms the semantic validity considers. Our
way of identifying valuational team semantics through LTP is very similar,
with the key extra step that the restriction to the classes of homomorphisms
is itself axiomatisable in LTP. 

The set of axioms we use for this purpose (PVA) identifies the class of
homomorphisms that map atomic formulas to principal ideals of the underlying
Boolean algebra. In the overlapping cases,\footnote{Pun{\v{c}}och{\'a}{\v{r}}  
develops his systems extending all intermediate logics, whereas our
presentation focuses on logics formed from classical logic and Boolean
algebras} this identifies exactly the same class of homomorphisms as 
Pun{\v{c}}och{\'a}{\v{r}} restricts his semantics to \cite{Puncochar2017,Puncochar2021}. 
However, since we are essentially able to encode the full valuational team
semantics in LTP for the highly expressive logic $\text{PT}^+$, our result also
applies to logics that are not limited to downward-closed properties of
teams.

Bezhanishvili et al. \cite{Bezhanishvili2021} establish a similar set of
axioms for the specific properties of atomic formulas in their
axiomatisations of inquisitive logics.
However, where our axiom set is based on a form of excluded middle 
$(p\ilor \slnot p)$, their axiomatisation instead uses a double negation
elimination, which in our language corresponds to the formula 
$(\slnot \slnot p\to p)$). 
As these sets of formulas often are equivalent, one might reasonably expect
axioms of the form $\uBox(\slnot \slnot p\to p)$ to be equivalent to the
PVA axioms. Indeed, if we only look at homomorphisms into the valuational
model $\P 2^{\mathbb{N}}$ they are equivalent \cite{LorimerOlsson2022}.
However, this equivalence holds only for complete Boolean algebras, and
therefore does not hold in general. To show this, we first need to understand
the algebraic properties of $\slnot$ as an operation on power algebras.

\begin{lemma}\label{lem: strict not props} For all Boolean algebras $B$ and all $A\in \P B$:
	\begin{enumerate}
		\item $b \in \slnot A$ if and only if $\bnot b$ is an upper bound of $A$.
		\item $b\in \slnot  A$ if and only if $\phantom{\bnot}b$ is a\phantom{n} lower bound of $A^{\bnot}= \set{\bnot a |a\in  A}$.
		\item	$ \slnot \slnot A$ is the set of lower bounds of the upper bounds of $A$, that is 
		$$ \slnot \slnot A = \set{b | b\leq c \text{ for all }  c  \text{ such that } a\leq c \text { for all } a\in A}, $$ 
		\item $\slnot \slnot  A$ is a principal ideal if and only if $\bigvee A$ exists in  $B$.
	\end{enumerate}	
\end{lemma}
\begin{proof}~
	\begin{enumerate}
		\item 
		First assume $b\in \slnot A$. Then $b\band a = \bbot$ for all $a\in A$, and thus
		$$a= a\band (b\bor\bnot b)=  \bbot \bor (a\band \bnot b)= a\band \bnot b$$
		This holds for all $a \in A$ and thus $\bnot b$ is an upper bound of $A$.
		
		For the other direction, assume $\bnot b$ is an upper bound of $A$. Then for all $a\in A$ 
		$$a =a \band ( b \bor \bnot b) = (a\band b)  \bor (a \band \bnot b)= (a \band b) \bor a,$$ 
		proving that $a\band  b \leq  a$. Since $-b$ is an upper bound for all $a$ we also conclude $$(a\band b) \leq (a\band b)\band - b= \bbot.$$  Thus, for all $a\in A$ we have that  $  a\band b = \bbot$ proving that $ b \in \slnot A$.
		\item 
		Observe that if $a\leq b$ for all $a\in A$, then $\bnot b \leq \bnot a$ for all $a\in A$, proving that $b$ is a lower bound for $A^\bnot$ exactly when $\bnot b$ is an upper bound for $A$.
		\item Follows directly from (1) and (2):
		$$\slnot \slnot A= \set{b| b\leq c\text{ for all } c \in \set{ \bnot d | d\in \slnot A} }$$ 
		$$= \set{b| b\leq c\text{ for all } c \in \set{ \bnot d | a\leq \bnot d \text { for all } a\in A} }$$ 
		$$= \set{b | b\leq c \text{ for all }  c \text{ such that } a\leq c \text { for all } a\in A}.$$
		\item
			First assert that if a least upper bound of a set $A$ exists it is equivalent to the least upper bound of the lower bounds of the upper bounds of the set. By (3) we can thus conclude that $\bigvee A= \bigvee \slnot \slnot A$ and that they exist simultaneously. 
			Assume $\dot{a} =  \bigvee A$ exists in  $B$. Then $\dot{a}$ is a lower bound of all upper bounds of $A$, and by (3) $\dot{a}\in \slnot \slnot A$ and $\slnot \slnot A= \set{ b | b\leq \dot{a}}$ is a principal ideal. On the other hand, assuming $\slnot \slnot A$ is a principal ideal, then there is some $c\in \slnot \slnot A$ such that $d\leq c$ for all $d\in \slnot\slnot A$. Clearly then $c= \bigvee \slnot \slnot A$ and thus $\bigvee A$ exists in $B$.
	\end{enumerate}
\end{proof}

\begin{prop}\label{negation elim} 
	For all Boolean algebras $B$, all homomorphisms $ H: \Fm \to \P B$, and all
	formulas $\phi \in \Fm$ we have $$ \vDash_H \slnot \slnot \phi \to \phi
	\quad \text{iff}\quad \slnot \slnot H(\phi)= H(\phi). $$
\end{prop}

\begin{proof} 
By Proposition \ref{prop_strict_neg} (1) and the substitutionality of LTP, we
have that $\vDash \phi\to \slnot \slnot \phi$ and it is clear that the
proposition follows from how $\to$ and $\slnot$ are defined.
\end{proof}
\begin{cor}
	For all atomic formulas  $P$, we have in LTP that $$\uBox
	(\slnot \slnot P \to P)\nvDash \uBox (P\ilor \slnot P ).$$ 
\end{cor}
\begin{proof}
	By Lemma \ref{lem: strict not props} (4) together with Theorem \ref
	{excluded_middle} and Proposition \ref{negation elim} it is clear that
	any non-complete Boolean algebra serves as a counterexample to the
	entailment. 
\end{proof}
\section{Conclusion}\label{section:Conclusion}

In this paper, we have introduced a new substitutional logic of team properties, LTP,
with a natural semantics inspired by algebraic semantics together with a
sound and complete labelled natural deduction system. Additionally, we
presented an axiomatisation of the propositional dependence logic $\text
{PT}^+$ within the framework of LTP. 

By adopting an algebraic perspective from the outset, the development of the
semantics, the natural deduction system, and the relative axiomatisation
appears both straightforward and natural in that 
the resulting structures
closely reflect the intended semantic meaning of formulas rather than being shaped
by contingent choices at the level of encoding. 
This coherence suggests that a structural analysis of these
constructions can yield valuable insights into team logics from an algebraic
standpoint. Finally, by focusing on different components of our framework, we
outline several research directions that emerge naturally from this work.


The logic LTP is fully substitutional, making it possible to apply the
techniques offered by the algebraic study of logics. More precisely, in the
terminology of abstract algebraic logic~\cite{Font16}, and as described in the end of
Section \ref{sec:def-class}, LTP is a semilattice-based logic with an
assertional companion $\text{LTP}^\etop$, both satisfying deduction theorems.
We believe that both LTP and $\text{LTP}^\etop$ should be analysed further from
the perspective of abstract algebraic logic in future work, further
categorising the logics in an established framework.

To relate LTP to other propositional team logics we have identified a set of
principal variable axioms (PVA). This set constitutes a natural candidate for
axiomatising the denotational semantics of the valuational team logic PT$^+$
within the semantic framework of LTP, analogous to how classical propositional
logic is embedded within PT$^+$. It is therefore expected that similar
algebraic constructions and axiomatisations are possible for other types of
team semantics such as modal team semantics
\cite{Vaeaenaenen2008}. Furthermore, this axiomatisation provides a way to
construct proofs of the entailment statements of these propositional team
logics. It does not, however, directly constitute a natural deduction system for
the axiomatised logics per se, since the terms of these proofs will in general
not be confined to the syntactic fragment of the logics. Our natural
deduction system may, however, motivate and guide the
construction of deduction systems for these propositional team logics, and
indicates the suitability of labelled systems.

In the labelled natural deduction for LTP, the rules $\sub$ and $\taut$
establish a notion of equivalence of labels determined by classical
propositional logic. From this perspective these rules can be viewed as
structural rules of the deduction system. The rules for the formulas of LTP
consist of introduction rules together with the elimination rules that are
the direct inverses of the introduction rules (up to equivalence of
labels).\footnote{The rule $\ine$ is not directly the inverse of $\ini$ but
can be seen to be equivalent to a direct inverse rule. } In this sense, the
rules of LTP harmonise, and it is possible to use more advanced
proof-theoretic methods to investigate LTP. For example, it seems to be easy
to turn the system into a sequent system that could be analysed with respect
to cut rules and cut elimination. This analysis may lead up to a
proof-theoretic explanation of the internal connectives and so also of the
connectives in other propositional team logics through their
axiomatisations.

By interpreting the $\taut$ rule as expressing provability in an inner logic,
we can generalise the construction of the labelled natural deduction system
into a proof-theoretic method for assigning team-based semantics to a logic.
This perspective can be understood as combining two layers: an inner logic
and an outer logic. The logic LTP then arises as a special case in which both
layers are instances of classical propositional logic. This 
opens the door to a purely proof-theoretic approach to team logics and their
broader analogues, which we see as a promising direction for future
research. Notably, it suggests the potential for analogous constructions in a
first-order setting that could give new general insight
into first-order team logics, which is an active field of study with many
applications.

\subsection*{Funding}

This paper was written as part of the project: \textit{Foundations for team
semantics: Meaning in an enriched framework}, a research project supported by
grant 2022-01685 of the Swedish Research Council, Vetenskapsrådet.

\bibliographystyle{alpha}
\bibliography{refs}

@misc{Knudstorp2025,
      title={Diamonds and Dominoes: Impossibility Results for Associative Modal Logics}, 
      author={S{\o}ren Brinck Knudstorp},
      year={2025},
      eprint={2506.16366},
      archivePrefix={arXiv},
      primaryClass={math.LO},
      url={https://arxiv.org/abs/2506.16366}, 
      note = {arXiv:2506.16366 [math.LO]},

}

@article{vaananen2010dependence,
  title={Dependence of variables construed as an atomic formula},
  author={V{\"a}{\"a}n{\"a}nen, Jouko and Hodges, Wilfrid},
  journal={Annals of Pure and Applied Logic},
  volume={161},
  number={6},
  pages={817--828},
  year={2010},
  publisher={Elsevier}
}

@article{gradel2013dependence,
  title={Dependence and independence},
  author={Gr{\"a}del, Erich and V{\"a}{\"a}n{\"a}nen, Jouko},
  journal={Studia Logica},
  volume={101},
  number={2},
  pages={399--410},
  year={2013},
  publisher={Springer}
}

@book{Font16,
	title={Abstract algebraic logic: An introductory textbook},
	author={Font, Josep Maria},
	year={2016},
	publisher={College Publications},
  address={London}
}

@article{Quadrellaro2021,
	title={On intermediate inquisitive and dependence logics: An algebraic study},
	author={Quadrellaro, Davide Emilio},
	journal={Annals of Pure and Applied Logic},
	volume={173},
	number={10},
	pages={103143},
	year={2022},
	publisher={Elsevier}
}

@Article{Bezhanishvili2021,
  author    = {Bezhanishvili, Nick and Grilletti, Gianluca and Quadrellaro, Davide Emilio},
  volume	= {15},
  DOI		= {10.1017/S175502032100054X},
  number	= {4},
  journal   = {The Review of Symbolic Logic},
  title     = {An algebraic approach to inquisitive and \texttt{DNA}-logics},
  year      = {2022},
  pages     = {950--990},
  publisher = {Cambridge University Press},
}

@Article{Quadrellaro2020,
  author  = {Quadrellaro, Davide Emilio},
  journal = {Short Papers Advances in Modal Logic AiML 2020},
  title   = {Algebraic Semantics of Intuitionistic Inquisitive and Dependence Logic},
  year    = {2020},
  pages   = {75},
}

@PhdThesis{Lueck2020,
  author = {L{\"u}ck, Martin},
  school = {Hannover: Institutionelles Repositorium der Leibniz Universit{\"a}t Hannover},
  title  = {Team logic: axioms, expressiveness, complexity},
  year   = {2020},
}

@Article{Hodges1997,
  author    = {Hodges, Wilfrid},
  journal   = {Logic Journal of the IGPL},
  title     = {Compositional semantics for a language of imperfect information},
  year      = {1997},
  number    = {4},
  pages     = {539--563},
  volume    = {5},
  publisher = {OUP},
}

@Article{Kontinen2013,
  author    = {Kontinen, Juha and V{\"a}{\"a}n{\"a}nen, Jouko},
  journal   = {Annals of Pure and Applied Logic},
  title     = {Axiomatizing first-order consequences in dependence logic},
  year      = {2013},
  number    = {11},
  pages     = {1101--1117},
  volume    = {164},
  doi       = {10.1016/j.apal.2013.05.006},
  publisher = {Elsevier},
}

@Book{Vaeaenaenen2007,
  author    = {V{\"a}{\"a}n{\"a}nen, Jouko},
  publisher = {Cambridge University Press},
  address = {Cambridge},
  title     = {Dependence Logic: A new approach to Independence Friendly Logic},
  year      = {2007},
  volume    = {70},
}

@article{Vaeaenaenen2008,
	title={Modal dependence logic},
	author={V{\"a}{\"a}n{\"a}nen, Jouko},
	journal={New perspectives on games and interaction},
	volume={4},
	pages={237--254},
	year={2008},
	publisher={Amsterdam University Press Amsterdam}
}

@MastersThesis{LorimerOlsson2022,
  author = {Lorimer Olsson, Orvar},
  school = {University of Gothenburg},
  title  = {Monadic Semantics, Team Logics and Substitution},
  year   = {2022},
  url    = {https://hdl.handle.net/2077/72005},
}

@InCollection{Hintikka1989,
  author     = {Hintikka, Jaakko and Sandu, Gabriel},
  booktitle  = {Logic, methodology and philosophy of science, {VIII} ({M}oscow, 1987)},
  publisher  = {North-Holland},
  title      = {Informational independence as a semantical phenomenon},
  year       = {1989},
  address    = {Amsterdam},
  pages      = {571--589},
  series     = {Stud. Logic Found. Math.},
  volume     = {126},
  doi        = {10.1016/S0049-237X(08)70066-1},
  mrclass    = {03B65 (68S10 68T30)},
  mrnumber   = {1034575 (90j:03054)},
  mrreviewer = {B. H. Mayoh},
}

@article{Yang2022,
	title={Propositional union closed team logics},
	author={Yang, Fan},
	journal={Annals of Pure and Applied Logic},
	volume={173},
	number={6},
	pages={103102},
	year={2022},
	publisher={Elsevier}
}

@Article{Yang2017,
  author    = {Yang, Fan and V{\"a}{\"a}n{\"a}nen, Jouko},
  journal   = {Annals of Pure and Applied Logic},
  title     = {Propositional team logics},
  year      = {2017},
  number    = {7},
  pages     = {1406--1441},
  volume    = {168},
  publisher = {Elsevier},
}

@article{Yang2016,
  title={Propositional logics of dependence},
  author={Yang, Fan and V{\"a}{\"a}n{\"a}nen, Jouko},
  journal={Annals of Pure and Applied Logic},
  volume={167},
  number={7},
  pages={557--589},
  year={2016},
  publisher={Elsevier}
}

@book{Bolzano1837,
  author     = {Bolzano, Bernard},
  title      = {Theory of Science},
  translator = {Rusnock, Paul and George, Rolf},
  volumes    = {4},
  publisher  = {Oxford University Press},
  address    = {Oxford},
  year       = {2014},
  note       = {Original work published 1837 as \emph{Wissenschaftslehre}}
}

@Book{Halmos2009,
  author    = {Halmos, Paul and Givant, Steven},
  publisher = {Springer},
  address = {Berlin},
  title     = {Introduction to Boolean algebras},
  year      = {2009},
}

@article{Goranko99,
	title={Hyperboolean algebras and hyperboolean modal logic},
	author={Goranko, Valentin and Vakarelov, Dimiter},
	journal={Journal of Applied Non-Classical Logics},
	volume={9},
	number={2-3},
	pages={345--368},
	year={1999},
	publisher={Taylor \& Francis}
}

@article{Brink1984,
	title={Second-order Boolean algebras},
	author={Brink, Chris},
	journal={Quaestiones Mathematicae},
	volume={7},
	number={2},
	pages={93--100},
	year={1984},
	publisher={Taylor \& Francis}
}

@article{Brink1986,
	title={Power structures and logic},
	author={Brink, Chris},
	journal={Quaestiones Mathematicae},
	volume={9},
	number={1-4},
	pages={69--94},
	year={1986},
	publisher={Taylor \& Francis}
}

@article{Brink1993,
	title={Power structures},
	author={Brink, Chris},
	journal={Algebra Universalis},
	volume={30},
	pages={177--216},
	year={1993},
	publisher={Springer}
}

@article{Goldblatt1989,
	title={Varieties of complex algebras},
	author={Goldblatt, Robert},
	journal={Annals of Pure and Applied Logic},
	volume={44},
	number={3},
	pages={173--242},
	year={1989},
	publisher={Elsevier}
}

@incollection{Priest2017,
	title={Plurivalent logics},
	author={Priest, Graham},
	journal={The logical legacy of Nikolai Vasiliev and modern logic},
	pages={169--179},
	year={2017},
	publisher={Springer}
}

@article{Humberstone2014,
	title={Power matrices and Dunn--Belnap semantics: Reflections on a remark of Graham Priest},
	author={Humberstone, Lloyd},
	journal={The Australasian Journal of Logic},
	volume={11},
	number={1},
	year={2014}
}

@incollection{Jonsson1993,
	title={A survey of Boolean algebras with operators},
	author={J{\'o}nsson, Bjarni},
	booktitle={Algebras and orders},
	pages={239--286},
	year={1993},
	publisher={Springer},
  address = {Berlin}
}

@incollection{Venema2007,
	title={6 Algebras and coalgebras},
	author={Venema, Yde},
	booktitle={Studies in Logic and Practical Reasoning},
	volume={3: Handbook of Modal Logic},
	pages={331--426},
	year={2007},
	publisher={Elsevier},
  address={Amsterdam}
}

@Inbook{Bull1984,
author="Bull, Robert and Segerberg, Krister",
editor="Gabbay, D. and Guenthner, F.",
title="Basic Modal Logic",
bookTitle="Handbook of Philosophical Logic: Volume II: Extensions of Classical Logic",
year="1984",
publisher="Springer Netherlands",
address="Dordrecht",
pages="1--88",
isbn="978-94-009-6259-0",
doi="10.1007/978-94-009-6259-0_1",
}

@article{Hella2023,
  title={Dimension in team semantics},
  author={Hella, Lauri and Luosto, Kerkko and V{\"a}{\"a}n{\"a}nen, Jouko},
  journal={Mathematical Structures in Computer Science},
  volume={34},
  number={5},
  pages={410--454},
  year={2024},
  publisher={Cambridge University Press}
}

@article{Puncochar2021,
	title={Inquisitive Heyting Algebras},
	author={Pun{\v{c}}och{\'a}{\v{r}}, V{\'\i}t},
	journal={Studia Logica},
	volume={109},
	number={5},
	pages={995--1017},
	year={2021},
	publisher={Springer}
}

@article{Puncochar2017,
	title={Algebras of information states},
	author={Pun{\v{c}}och{\'a}{\v{r}}, V{\'\i}t},
	journal={Journal of Logic and Computation},
	volume={27},
	number={5},
	pages={1643--1675},
	year={2017},
	publisher={Oxford University Press}
}

@article{Ciardelli2011,
	title={Inquisitive logic},
	author={Ciardelli, Ivano and Roelofsen, Floris},
	journal={Journal of Philosophical Logic},
	volume={40},
	number={1},
	pages={55--94},
	year={2011},
	publisher={Springer}
}

\end{document}